\documentclass[11pt]{amsart} \textwidth=14.5cm \oddsidemargin=1cm
\usepackage{amsmath,wasysym}
\usepackage{amsxtra}
\usepackage{amscd}
\usepackage{amsthm}
\usepackage{amsfonts}
\usepackage{amssymb}
\usepackage{eucal}
\usepackage{graphics, color}
\usepackage{hyperref,mathrsfs}
\usepackage[usenames,dvipsnames]{xcolor}
\usepackage{tikz}
\usepackage{stmaryrd}
\usepackage[all]{xy}
\usepackage{hyperref}
\usepackage{color}
\usepackage{bbm} 
\allowdisplaybreaks

\hypersetup{colorlinks,linkcolor=blue, urlcolor=blue,
citecolor=blue}

\newtheorem{thm}{Theorem} [section]
\newtheorem{lem}[thm]{Lemma}
\newtheorem{cor}[thm]{Corollary}
\newtheorem{prop}[thm]{Proposition}

\newtheorem{rem}[thm]{Remark}

\numberwithin{equation}{section}
\newcommand{\A}{\mathcal{A}}
\newcommand{\mA}{\mathbf{A}}
\newcommand{\mB}{\mathbf{B}}
\newcommand{\sB}{\mathscr{B}}
\newcommand{\co}{\mathrm{col}}

\newcommand{\mC}{\mathbf{C}}
\newcommand{\D}{\mathcal{D}}

\newcommand{\diag}{\mathrm{diag}}

\newcommand{\ve}{\varepsilon}
\newcommand{\End}{\mathrm{End}}
\newcommand{\F}{\mathbb{F}}
\newcommand{\ff}{\mathtt{f}}
\newcommand{\sF}{\mathscr{F}}

\newcommand{\Tfg}{\T_{\ff\g}}
\newcommand{\g}{\mathtt{g}}

\newcommand{\HH}{\mathcal{H}}
\newcommand{\Hom}{\mathrm{Hom}}

\newcommand{\K}{\mathbb{K}}
\newcommand{\Ker}{\mathrm{Ker}}

\newcommand{\mi}{\mathbf{i}}

\newcommand{\Par}{\mathrm{Par}}
\newcommand{\N}{\mathbb{N}}
\newcommand{\Ob}{\mathscr{O}}

\newcommand{\ro}{\mathrm{row}}
\newcommand{\qU}{\mathbf{U}}

\newcommand{\q}{\mathbf{q}}
\newcommand{\Sc}{\mathcal{S}}

\newcommand{\T}{\mathbb{T}}

\newcommand{\Z}{\mathbb{Z}}
\newcommand{\half}{\frac{1}{2}}

\title[Geometric Howe dualities]
{Geometric Howe dualities of finite type}

\author[Li Luo]{Li Luo}
\author[Zheming Xu]{Zheming Xu}
\address{School of mathematical Sciences,
Shanghai Key Laboratory of Pure Mathematics and Mathematical Practice,
    East China Normal University, Shanghai 200241, China}
\email{lluo@math.ecnu.edu.cn (Luo), 51195500019@stu.ecnu.edu.cn (Xu)}

\begin{document}

\begin{abstract}
We develop a geometric approach toward an interplay between a pair of quantum Schur
algebras of arbitrary finite type. Then by Beilinson-Lusztig-MacPherson's stabilization
procedure in the setting of partial flag varieties of type A (resp. type B/C), the Howe
duality between a pair of quantum general linear groups (resp. a pair of
$\imath$quantum groups of type AIII/IV) is established. The Howe duality for quantum
general linear groups has been provided via quantum coordinate algebras in \cite{Z02}.
We also generalize this algebraic approach to $\imath$quantum groups of type AIII/IV,
and prove that the quantum Howe duality derived from partial flag varieties coincides
with the one constructed by quantum coordinate (co)algebras. Moreover, the explicit
multiplicity-free decompositions for these Howe dualities are obtained.
\end{abstract}

\maketitle
\setcounter{tocdepth}{1}
\tableofcontents

\section{Introduction}
\subsection{}
The classical Howe duality, which involves commuting actions of a pair of Lie
groups/algebras, provides a representation theoretical treatment for classical invariant
theory \cite{Ho89}. For the general linear Lie groups/algebras, Schur duality, Howe
duality, and the first fundamental theorem (FFT) are equivalent.

Nowadays, there have been a number of quantum versions of Howe dualities. The first one
was achieved by Quesne \cite{Q92} on the duality between quantum groups
$U_q(\mathfrak{su}_3)$ and $U_q(\mathfrak{u}_2)$. Noumi, Umeda and Wakayama
\cite{NUW95,NUW96} obtained quantum analogues of the dual pairs
$(\mathfrak{sl}_2,\mathfrak{so}_n)$ and $(\mathfrak{sp}_2,\mathfrak{so}_n)$, where not
the usual Drinfeld-Jimbo quantum group $U_q(\mathfrak{so}_n)$ but another different
$q$-deformation $U_q'(\mathfrak{so}_n)$ of the universal enveloping algebra
$U(\mathfrak{so}_n)$ was involved. The Howe duality for a pair of quantum general linear
groups $U_q(\mathfrak{gl}_m)$ and $U_q(\mathfrak{gl}_n)$ was given by Zhang in
\cite{Z02}, where quantum coordinate algebras were employed to construct a
non-commutative analogue of the symmetric algebras on which $U_q(\mathfrak{gl}_m)$ and
$U_q(\mathfrak{gl}_n)$ act. This construction was further applied to established the Howe
duality of $(U_q(\mathfrak{gl}_n), U_q(\mathfrak{so}_{2n}))$, $(U_q(\mathfrak{gl}_n),
U_q(\mathfrak{so}_{2n+1}))$ and $(U_q(\mathfrak{gl}_n), U_q(\mathfrak{sp}_{2n}))$ in
\cite{LZ03} (see also \cite{WZ09,CW20} for quantum supergroups). It also helps to provide
a non-commutative version of the FFT for associated  quantum groups (cf.
\cite{LZZ11,Zh20}). Another special quantum Howe duality construction for type A can be
found in \cite{FKZ19}.

\subsection{}
As we saw in the aforementioned papers \cite{NUW95,NUW96}, a nonstandard $q$-deformation
$U_q'(\mathfrak{so}_n)$ occurs. This is no isolated instance. In \cite{ES18,ST19},
nonstandard $q$-deformations are used to set up other quantum Howe dualities, too. Actually, all
of them are examples of another family of $q$-deformations of the universal enveloping
algebras of complex Lie algebras, called $\imath$quantum groups. An $\imath$quantum
group $\mathbf{U}^\imath$ is a coideal subalgebra of a quantum group
$\mathbf{U}=U_q(\mathfrak{g})$ of a simple complex Lie algebra $\mathfrak{g}$ such that
$(\mathbf{U},\mathbf{U}^\imath)$ forms a quantum symmetric pair, which was introduced by
Letzter \cite{Le99}. The classification of quantum symmetric pairs (and hence of
$\imath$quantum groups) can be described by Satake diagrams.

In their remarkable work \cite{BW18}, Bao and Wang used the $\imath$quantum groups,
associated with the Satake diagram of type AIII/IV with no black node, to reformulate the
Kazhdan-Lusztig theory of type B/C without using Hecke algebras directly, and then
provided an elegant conceptual solution to the problem of irreducible characters for
$\mathfrak{osp}$ type Lie superalgbras. This successful application of $\imath$quantum
groups initiates the ``$\imath$-program'': generalizing various achievements for quantum
groups to $\imath$quantum groups.

\subsection{}
In their 1990 paper \cite{BLM90}, Beilinson, Lusztig and MacPherson (BLM) gave a
geometric realization of $U_q(\mathfrak{gl}_n)$ and the canonical basis via partial flag
varieties of type A. Their first step is to obtain the quantum Schur algebra
$\mathcal{S}_{n,d}$ on pairs of $n$-step partial flags in a $d$-dimensional space as a
convolution algebra, which imitates Iwahori's geometric realization of the Hecke algebra
$\mathcal{H}_d$ on pairs of complete flags (cf. \cite{IM65}). Then they derived a
stabilization property from some closed multiplication formulas of $\mathcal{S}_{n,d}$.
This stabilization property helps to construct a bigger algebras in which
$U_q(\mathfrak{gl}_n)$ embeds.

Moreover, a Fock space, equipped with a left (resp. right) action of $\mathcal{S}_{n,d}$
(resp. $\mathcal{H}_d$), can be realized by using a pair of a partial flag and a complete
flag (see \cite{GL92}). This Fock space admits a double centralizer property between
$\mathcal{S}_{n,d}$ and $\mathcal{H}_d$. As a by-product, the celebrated Schur-Jimbo
duality \cite{Jim86} between $U_q(\mathfrak{gl}_n)$ and $\mathcal{H}_d$ is rediscovered
thanks to BLM's stabilization procedure.

The above geometric realization has been generalized to $\imath$quantum groups
\cite{BKLW18, FL15} in the setting of partial flag varieties of type B/C/D (see also
\cite{FLLLW20} for affine type C). Furthermore, the first author and Wang \cite{LW22}
generalized the notion of $n$-step partial flag variety to arbitrary finite type and
then introduced quantum Schur algebras of arbitrary finite type in terms of partial
flags similar to BLM construction.

In the aforementioned double centralizer property between $\mathcal{S}_{n,d}$ and
$\mathcal{H}_d$, Wang \cite{W01} replaced the Fock space by another one that is defined
by using a pair of an $m$-step partial flag and an $n$-step partial flag. By similar
arguments to those in \cite{GL92}, this new Fock space admits a double centralizer
property between two quantum Schur algebras $\mathcal{S}_{m,d}$ and $\mathcal{S}_{n,d}$,
which immediately implies a double centralizer property between $U_q(\mathfrak{gl}_m)$
and $U_q(\mathfrak{gl}_n)$ thanks to BLM's stabilization procedure again. This
observation was also achieved independently by Baumann in an unpublished paper
\cite{Ba07}.

\subsection{}

Our first main result is a general geometric construction of dualities between a pair of
quantum Schur algebras for arbitrary finite type via partial flag varieties (see
Theorem~\ref{duality over Z[q,q{-1}]}), which is a generalization of
$(\mathcal{S}_{m,d}, \mathcal{S}_{n,d})$-duality given in \cite{W01,Ba07}. Such a geometric construction helps us construct a canonical basis of the Fock space by a standard way as in \cite[\S1.4]{BLM90}. Then the positivity of the actions of quantum Schur algebras on the Fock space is derived by a standard geometric argument (see Theorem~\ref{positivity}).

For type A, although Baumann has lifted the above $(\mathcal{S}_{m,d},
\mathcal{S}_{n,d})$-duality to the Howe duality between $U_q(\mathfrak{gl}_m)$ and
$U_q(\mathfrak{gl}_n)$ by BLM's stabilization procedure, we compute the explicit formulas
of the left $U_q(\mathfrak{gl}_m)$-action and right $U_q(\mathfrak{gl}_n)$-action on the
Fock space. We also compute the explicit action formulas of $(U_q(\mathfrak{gl}_m),
U_q(\mathfrak{gl}_n))$-duality via Zhang's quantum coordinate algebras approach. All
these formulas are new. They help us show that Baumann's geometric approach and Zhang's
algebraic approach of $(U_q(\mathfrak{gl}_m), U_q(\mathfrak{gl}_n))$-duality are
equivalent (see Theorem~\ref{equivalent}).

For type B/C, we establish the Howe duality between a pair of $\imath$quantum groups (called an \emph{$\imath$Howe duality}) via
both geometric approach (see
Theorem~\ref{doubleAIV}) and algebraic approach, and prove that they coincide (see Theorems \ref{equivalentj}). In contrast to type A, the (quasi) quantum matrix spaces for type B/C
admit no multiplication but only the comultiplication since there is no comultiplication
on the $\imath$quantum groups of type AIII/IV. Therefore for $\imath$quantum groups we
use the notion of quantum coordinate coalgebra instead. We show that our quantum coordinate coalgebras coincide with the ones introduced by Lai-Nakano-Xiang \cite{LNX22}.

We obtain the multiplicity-free decomposition of the Fock space appeared in the $\imath$Howe duality (see Theorem~\ref{decomp}). The
formulation employs the classical weight module theory of $\imath$quantum groups
established by Watanabe in his recent work \cite{Wa21}.

\subsection{}
The paper is organized as follows. Section \ref{section2} is devoted to the Howe
dualities between a pair of quantum Schur algebras for arbitrary finite type in the sense
of \cite{LW22}. We specialize the general construction at type
A in Section \ref{section3}. Both geometric approach (in terms of flag varieties) and algebraic approach (in term
of coordinate algebras) of $(U_q(\mathfrak{gl}_m), U_q(\mathfrak{gl}_n))$-duality are
revisited. We formulate explicit actions of $U_q(\mathfrak{gl}_m)$ and
$U_q(\mathfrak{gl}_n)$ via both approaches and prove that these two coincide.
In Section \ref{section4}, we establish $\imath$Howe dualities in the setting of partial flag varieties of type B/C. In Section \ref{section5}, we provide an
algebraic construction via quantum coordinate coalgebras, which also coincides with the
geometric one as the same as type A. Finally, we provide the multiplicity-free decomposition of the Fock space appearing in the $\imath$Howe duality in Section \ref{section6}.

\subsection*{Acknowledgement}
We thank Weiqiang Wang for proposing the topic and providing many helpful ideas selflessly. We also thank Hideya Watanabe for explanation on classical weight modules of $\imath$quantum groups, and thank Runqiang Jian for the knowledge of Bruhat cells. We would like to express our gratitude to the referee for the insightful comments towards several improvements of Section 2.

LL is partially supported by the Science and Technology Commission of Shanghai Municipality (grant No. 22DZ2229014, 21ZR1420000) and the NSF of China (grant No. 11871214).


\section{General construction for arbitrary finite type}\label{section2}

\subsection{Weyl group orbits on weight lattice}
Let $G_\Z$ be a split and connected reductive algebraic $\Z$-group, $T_\Z$ a split
maximal torus of $G_\Z$. Let $W$ be the Weyl group of $G_\Z$ associated with $T_\Z$ and
$X$ be its weight lattice. Fix a simple system $\Pi=\{\alpha_1,\ldots,\alpha_d\}$. Then
$W$ is generated by the simple reflections $s_1,\ldots,s_d$. Let us take two W-invariant
finite subsets
\begin{equation*}
X_\ff, X_\g \subset X.
\end{equation*}
Denote
\begin{align*}
\Lambda =\{\mbox{$W$-orbits in $X$}\},\quad
\Lambda_\ff =\{\mbox{$W$-orbits in $X_\ff$}\}, \quad \Lambda_\g=\{\mbox{$W$-orbits in
$X_\g$}\}.
\end{align*}
Note that in each $W$-orbit $\gamma \subset X$, there exists a unique anti-dominant
element which will be denoted by $\mathbf{i}_\gamma$.

For any subset $J \subset \{1,2,\ldots,d \}$, let $W_J$ be the parabolic subgroup of $W$
generated by $\{s_j \mid j \in J\}$. For any $W$-orbit $\gamma \in \Lambda$, we define
the subset
\begin{equation} \label{eq:PP}
J_\gamma= \{k  ~|~  1 \leq k \leq d, \mi_\gamma s_k= \mi_\gamma\}.
\end{equation}
We shall write $W_\gamma=W_{J_\gamma}$.

Let
$\D_\gamma=\{v \in W  ~|~  \ell(wv)= \ell(w)+ \ell(v), \forall w \in W_\gamma\}.$
Then $\D_\gamma$ (resp. $\D_\gamma^{-1}$) is the set of distinguished minimal length
right (resp. left) coset representatives of $W_\gamma$ in $W$.
Denote by
$\D_{\gamma\nu}=\D_\gamma^{-1} \cap \D_\nu$
the set of minimal length double coset representatives of $W_\gamma\setminus W/W_\nu$.
\subsection{Flag varieties}\label{flv}
Let $B_\Z$ be the Borel subgroup of $G_\Z$ corresponding to $\Pi$, and $\F$ a field. Let
$$G=G_\Z (\F), \quad B=B_\Z (\F) \quad\mbox{and}\quad T=T_\Z (\F)$$
be the sets of $\F$-valued points of $G_{\Z}$, $B_\Z$ and $T_\Z$, respectively.

It is known that $W \cong N_G(T)/T$ where $N_G(T)$ is the normalizer of $T$ in $G$. For any $w \in W$, let us choose a representative (still denoted by $w$)
in $G$ of $N_G(T)/T$. Associated to each subset $J$, we have a standard parabolic
subgroup $P_J=BW_JB$ which contains $B$. In particular,
$W_\emptyset=\{ 1 \}$ and hence $P_\emptyset=B$.
For $\gamma \in \Lambda$, we shall denote $P_\gamma=P_{J_\gamma}=BW_\gamma B$.

Denote by $\sB=G/B$ the complete flag variety.
We shall consider another two partial flag varieties:
\begin{equation*}
\sF_\ff= \bigsqcup_{\gamma \in \Lambda_\ff} \sF_\gamma, \quad \sF_\g=\bigsqcup_{\gamma
\in \Lambda_\g} \sF_\gamma, \quad\mbox{where
$\sF_{\gamma}=G/P_{\gamma}$.}
\end{equation*}
Clearly there is a natural $G$-action on $\sF_\gamma$ and hence on $\sF_\ff$ and
$\sF_\g$. Let $G$ act diagonally on $\sF_\gamma\times \sF_\nu (\gamma,\nu \in \Lambda)$,
and so on $\sF_\ff \times \sF_\ff$, $\sF_\ff \times \sF_\g$ and $\sF_\g \times \sF_\g$,
respectively.

Denote
\begin{align*}
\Xi_\ff&=\bigsqcup_{\gamma,\nu \in \Lambda_\ff} \{\gamma\} \times \D_{\gamma\nu}
\times\{\nu\},\qquad
\Xi_{\ff\g}=\bigsqcup_{\gamma \in \Lambda_\ff,\nu \in \Lambda_\g} \{\gamma\} \times
\D_{\gamma\nu} \times \{\nu\}, \\
\Xi_\g&=\bigsqcup_{\gamma,\nu \in \Lambda_\g} \{\gamma\} \times \D_{\gamma\nu} \times
\{\nu\}.
\end{align*}
There is a bijection between $\D_{\gamma\nu}$ and the $G$-orbits $G \backslash
(\sF_\gamma \times \sF_\nu)$, which sends $w \in\D_{\gamma\nu}$ to the $G$-orbit
containing $(P_\gamma,wP_\nu)$. Hence, the $G$-orbits in $\sF_\ff \times \sF_\ff$ (resp.
$\sF_\ff \times \sF_\g$ and $\sF_\g \times \sF_\g$) can be indexed by $\Xi_\ff$ (resp.
$\Xi_{\ff\g}$ and $\Xi_\g$). The orbit related to $\xi\in \Xi_\ff$ or $\Xi_{\ff\g}$ or
$\Xi_\g$
will be denoted by $\Ob_\xi$.

\subsection{Convolution product}
Let $q$ be an indeterminant, and let
$$\A=\Z[q,q^{-1}].$$

We set
$$
\Sc_\ff=\A_G(\sF_\ff \times \sF_\ff), \qquad \Tfg=\A_G(\sF_\ff \times \sF_\g), \qquad
\Sc_\g=\A_G(\sF_\g \times \sF_\g)
$$
to be the spaces of $G$-invariant $\A$-valued functions on $\sF_\ff \times \sF_\ff$,
$\sF_\ff \times \sF_\g$ and $\sF_\g \times \sF_\g$, respectively. Moreover, let
$$\T_\ff=\A_G(\sF_\ff \times \sB),\quad \T_\g=\A_G(\sF_\g \times \sB),\quad
\T_\gamma=\A_G(\sF_\gamma \times \sB) \quad (\forall \gamma \in \Lambda).$$

There is a convolution product $*$ on $\Sc_\ff$ (and on $\Sc_\g$) defined as follows, which is an imitation
of the geometric realization of Hecke algebras $$\mathcal{H}=\A_G(\sB \times \sB)$$ due
to Iwahori (cf. \cite{IM65}). We take $\F=\F_{\mathbf{q}}$ the finite field with
$\mathbf{q}$ elements.
For a triple $(\xi,\xi',\xi'')$ in $\Xi_\ff \times \Xi_\ff \times \Xi_\ff$,
fix $(\mathfrak{f}_1,\mathfrak{f}_2) \in \Ob_{\xi''}$, and  let
$\kappa_{\xi,\xi',\xi'';\q}$ be the number of $\mathfrak{f} \in \sF_\ff$ such that
$(\mathfrak{f} _1,\mathfrak{f} ) \in \Ob_\xi$ and $(\mathfrak{f},\mathfrak{f}_2) \in
\Ob_{\xi'}$.
A well-known property (cf. \cite{BLM90}) implies that there exists a polynomial
$\kappa_{\xi,\xi',\xi''} \in \Z[q^{-2}]$ such that
$\kappa_{\xi,\xi',\xi'';\q}=\kappa_{\xi,\xi',\xi''}|_{q^{-2}=\q}$ for all prime powers
$\q$. Let $\chi_\xi$ be the characteristic function of the orbit
$\Ob_\xi$. We define the convolution product on $\Sc_\ff$ by letting
\begin{equation*}
\chi_\xi*\chi_{\xi'}=\sum_{\xi''} \kappa_{\xi,\xi',\xi''}\chi_{\xi''}.
\end{equation*}
Equipped with the convolution product, the $\A$-module $\Sc_\ff$ (similarly $\Sc_\g$)
becomes an associative $\A$-algebra, which is named a \emph{quantum Schur algebra} or \emph{$q$-Schur
algebra} in \cite{LW22}.

A convolution product analog for $\Sc_\ff$ (resp. $\Sc_\g$) by regarding
$(\xi,\xi',\xi'') \in \Xi_\ff \times \Xi_{\ff\g} \times \Xi_{\ff\g}$ (resp. $\Xi_{\ff\g}
\times \Xi_{\ff\g} \times \Xi_\g$) gives us
a left $\Sc_\ff$-action $\Phi$ (resp. right $\Sc_\g$-action $\Psi$) on $\Tfg$:
$$\Sc_\ff \quad \stackrel{\Phi}\curvearrowright \quad \Tfg \quad
\stackrel{\Psi}\curvearrowleft \quad \Sc_\g.$$
The two actions $\Phi$ and $\Psi$ commute by definition.

Here is a technical lemma about the convolution product.
\begin{lem}
	\label{arithmetics of special case}
	Let $\xi=(\gamma, w, \nu), \xi'=(\nu, 1, \mu)$ with $P_\mu\subset P_\nu$. Then
	\begin{align*}
		\chi_\xi*\chi_{\xi'} =\chi_{(\gamma, w, \mu)}+\sum_{w\neq\sigma \in \D_{\gamma\mu} \cap (W_\gamma w W_\nu)}a_\sigma \chi_{(\gamma, \sigma, \mu)}, \quad (a_\sigma\in\A).
	\end{align*}
\end{lem}
\begin{proof}
Let $g\in\D_{\gamma\mu}$ satisfy  $(P_\gamma, g'P_\nu) \in \Ob_\xi$ and $(g'P_\nu, gP_\mu) \in \Ob_{\xi'}$ for some $g'\in G$. Since $(P_\gamma, g'P_\nu)\sim(P_\gamma, wP_\nu)\in\Ob_\xi$, we have $g'\in P_\gamma w P_\nu$. Here and below we always write $X\sim Y$ to reveal that $X$ and $Y$ are in the same $G$-orbit. Since $(P_\nu, g'^{-1}gP_\mu)\sim(P_\nu,P_\mu)\in\Ob_{\xi'}$, we have $g'^{-1}g\in P_\nu P_\mu=P_{\nu}$, where $P_\nu P_\mu=P_\nu$ follows from the condition $P_\mu\subset P_\nu$. Therefore, $g\in g'P_\nu\subset  P_\gamma w P_\nu P_\nu=P_\gamma w P_\nu=BW_\gamma BwBw_\nu B\subset BW_\gamma w W_\nu B$, where the last inclusion ``$\subset$'' comes from the well known property about the product of Bruhat cells (cf. \cite[\S2]{Bo68}). So $(P_\gamma, gP_\mu)\sim(P_\gamma, \sigma P_\mu)$ for some $\sigma \in \D_{\gamma\mu} \cap (W_\gamma w W_\nu)$, and hence $\chi_\xi*\chi_{\xi'}\in\sum_{g \in \D_{\gamma\mu}}\A \chi_{(\gamma, g, \mu)}=\sum_{\sigma \in \D_{\gamma\mu} \cap (W_\gamma w W_\nu)}\A \chi_{(\gamma, \sigma, \mu)}$. We obtain $\chi_\xi*\chi_{\xi'}=\sum_{\sigma \in \D_{\gamma\mu} \cap (W_\gamma w W_\nu)}a_\sigma \chi_{(\gamma, \sigma, \mu)}$ for some $a_\sigma\in\A$.

Clearly, $w\in\D_{\gamma\nu}\subset\D_{\gamma\mu} \cap (W_\gamma w W_\nu)$ because of $P_{\mu}\subset P_{\nu}$. Now let us prove $a_{w}=1$ by counting the number of flags $\mathfrak{f}\in G/P_\nu$ such that $(P_\gamma, \mathfrak{f}) \in \Ob_\xi$ and $(\mathfrak{f}, wP_\mu) \in \Ob_{\xi'}$. Write $\mathfrak{f}=gP_\nu$ for some $g\in G$. We know $wP_\mu\subset wP_\nu$ because of the condition $P_\mu\subset P_\nu$. Meanwhile, since $(gP_\nu,wP_\mu)\sim(P_\nu,P_\mu)\in\Ob_{\xi'}$, we have $wP_\mu \subset gP_\nu$ by the condition $P_\mu\subset P_\nu$ again. Thus $gP_\nu\cap wP_\nu\supset wP_\mu \neq\emptyset$, which implies that $\mathfrak{f}=gP_\nu= wP_\nu$ is unique, i.e. $a_{w}=1$ as desired.
\end{proof}

Lemma~\ref{arithmetics of special case} immediately implies the following corollary, which will be employed twice in the proof of Theorem~\ref{duality over Z[q,q{-1}]}.
\begin{cor}
\label{coefficient}
  For any $\gamma,\nu\in\Lambda_\ff,\mu\in\Lambda_{\g}$ with $P_{\mu}\subset P_{\nu}$ and $w\in\D_{\gamma\nu}$, the coefficient of $\chi_{(\gamma,w,\nu)}$ in $\chi$ coincides with that of $\chi_{(\gamma,w,\mu)}$ in $\chi\ast\chi_{(\nu,1,\mu)}$.
\end{cor}

\subsection{A canonical basis}
Now assume that $\F$ is algebraically closed. A bar involution $\bar{}:\Sc_\ff \to \Sc_\ff$ and a canonical basis $\mathbf{B}(\Sc_\ff)$ of $\Sc_\ff$ were constructed
in \cite[\S 4.3]{LW22}. We shall introduce a canonical basis of $\T_{\ff\g}$ by a similar way in this subsection.

For any $\xi =(\gamma, w, \nu)\in \Xi_{\ff\g}$, we denote $[\xi]=q^{d(\xi)-r(\xi)}
\chi_\xi$ where $d(\xi)=\dim(\Ob_\xi)$ and $r(\xi)=\dim(G/P_{\gamma})$. Then $\{[\xi]~|~\xi\in\Xi_{\ff\g}\}$ forms basis for $\T_{\ff\g}$ (called a standard basis).

Let $\mathrm{IC}_\xi$ be the shifted intersection complex associated with
$\overline{\Ob}_\xi$ such that the restriction of $\mathrm{IC}_\xi$ to $\Ob_\xi$ is the
constant sheaf of dimension $1$ on $\Ob_\xi$. Let $\mathscr{H}_{\xi'}(\mathrm{IC}_\xi)$
denote the stalk of the $i$th cohomology group of $\mathrm{IC}_\xi$ at any point in
$\Ob_{\xi'}$ (for $\Ob_{\xi'} \subset \overline{\Ob}_\xi$). We set
\begin{align*}
\{\xi\}&=\sum_{\xi' \leq \xi} P_{\xi',\xi}[\xi']\quad\mbox{where}\quad P_{\xi',\xi}=\sum_{i \in \Z}
\dim{\mathscr{H}_{\xi'}(\mathrm{IC}_\xi)}q^{-i+d(\xi)-d(\xi')}.
\end{align*}
Here the partial order $<$ is the orbit closure order. That is, for $\xi=(\gamma,g,\nu)$ and $\xi'=(\gamma',g',\nu')$,
$$\xi'<\xi \quad \Leftrightarrow \quad \gamma'=\gamma, \nu'=\nu, g'<g.$$
The properties of intersection complexes imply that $P_{\xi,\xi}=1$ and $P_{\xi',\xi} \in q\N[q]$ for $\xi'<\xi$.
As in \cite[\S 1.4]{BLM90}, we have an anti-linear bar involution $\bar{}:\T_{\ff\g} \to \T_{\ff\g}$ such that
$$\overline{\{\xi\}} = \{\xi\}\quad \mbox{for any $\xi\in\Xi_{\ff\g}$}.$$ In particular,
$$\overline{[\xi]}=\sum_{\xi' \leq \xi} c_{\xi',\xi}[\xi'], \quad\mbox{where} \ c_{\xi,\xi}=1,
c_{\xi',\xi} \in \A.$$ Then $\mathbf{B}(\T_{\ff\g}):=\{\{\xi\} ~|~ \xi \in \Xi_\ff\}$ forms an $\A$-basis for
$\T_{\ff\g}$, called a canonical basis. The bar maps (on $\Sc_\ff$, $\T_{\ff\g}$ and $\Sc_\g$) are compatible with the commuting actions of $(\Sc_\ff,\Sc_\g)$ on $\T_{\ff\g}$.

\begin{thm}[Positivity property]\label{positivity}
For any $a\in\mathbf{B}(\Sc_\ff)$, $b\in\mathbf{B}(\T_{\ff\g})$ and $c\in\mathbf{B}(\Sc_\g)$, we write
$$a\cdot b=\sum_{b'\in\mathbf{B}(\T_{\ff\g})}m_{a,b}^{b'}b',\quad b\cdot c=\sum_{b'\in\mathbf{B}(\T_{\ff\g})}n_{b,c}^{b'}b',\quad\mbox{for $m_{a,b}^{b'},n_{b,c}^{b'}\in\A$}.$$
Then we must have $m_{a,b}^{b'},n_{b,c}^{b'}\in \mathbb{N}[q,q^{-1}]$.
\end{thm}
\begin{proof}
  This follows from the geometric interpretation of these canonical bases and their action in terms of perverse sheaves and their convolution products.
\end{proof}

\subsection{The $(\Sc_\ff,\Sc_\g)$-duality} \label{sec2.5}
Let $$\mathbb{A}=\K(q),$$
where $\K$ is an arbitrary field of characteristic $0$.
We will always add a subscript $\mathbb{A}$ on the bottom-left of an $\A$-module (or an
$\A$-map) to mean the base change $\mathbb{A} \otimes_\A -$, e.g. ${}_\mathbb{A}\!\HH$,
${}_\mathbb{A}\!\Sc_\ff$, ${}_\mathbb{A}\!\Tfg$, ${}_\mathbb{A}\!\Phi$, etc.

\begin{thm}
	\label{duality}
	The actions
	$${}_\mathbb{A}\!\Sc_\ff \quad \stackrel{{}_\mathbb{A}\!\Phi}\curvearrowright \quad
{}_\mathbb{A}\!\Tfg \quad \stackrel{{}_\mathbb{A}\!\Psi}\curvearrowleft \quad
{}_\mathbb{A}\!\Sc_\g$$
	satisfy the double centralizer property
	$${}_\mathbb{A}\!\Phi({}_\mathbb{A}\!\Sc_\ff) =
\End_{{}_\mathbb{A}\!\Sc_\g}({}_\mathbb{A}\!\Tfg), \quad
\End_{{}_\mathbb{A}\!\Sc_\ff}({}_\mathbb{A}\!\Tfg)={}_\mathbb{A}\!\Psi({}_\mathbb{A}\!\Sc_\g).$$
\end{thm}
\begin{proof}
It has been shown in \cite[Theorem~4.2]{LW22} that
	$$\Sc_\ff \cong \End_\HH(\T_\ff),
\quad \Sc_\g \cong \End_\HH(\T_\g),$$
and hence naturally
$${}_\mathbb{A}\!\Sc_\ff \cong \End_{{}_\mathbb{A}\!\HH}({}_\mathbb{A}\!\T_\ff),
\quad {}_\mathbb{A}\!\Sc_\g \cong \End_{{}_\mathbb{A}\!\HH}({}_\mathbb{A}\!\T_\g).$$

The Hecke algebra ${}_\mathbb{A}\!\HH$ is split semisimple when take $q$ an indeterminant, so are the quantum Schur algebras
${}_\mathbb{A}\!\Sc_\ff$, ${}_\mathbb{A}\!\Sc_\g$ and their quotients ${}_\mathbb{A}\!\Phi({}_\mathbb{A}\!\Sc_\ff)$, ${}_\mathbb{A}\!\Psi({}_\mathbb{A}\!\Sc_\g)$.
Hence $${}_\mathbb{A}\!\Tfg \cong
\Hom_{{}_\mathbb{A}\!\HH}({}_\mathbb{A}\!\T_\g,{}_\mathbb{A}\!\T_\ff) \cong
{}_\mathbb{A}\!\T_\ff \otimes_{{}_\mathbb{A}\!\HH} {}_\mathbb{A}\!\T_\g^*$$
and $${}_\mathbb{A}\!\T_\ff\cong\bigoplus_{i} V_i\otimes M_i\quad\mbox{and}\quad {}_{\mathbb{A}}\!\T_\g\cong\bigoplus_{i} U_i\otimes M_i$$
where $V_i$'s (resp. $U_i$'s) are all left simple ${}_\mathbb{A}\!\Phi({}_\mathbb{A}\!\Sc_\ff)$-modules (resp. ${}_\mathbb{A}\!\Psi({}_\mathbb{A}\!\Sc_\g)$-modules) and $M_i$'s are certain right simple ${}_\mathbb{A}\!\HH$-modules up to an isomorphism.
Compute that
$${}_\mathbb{A}\!\Tfg\cong{}_\mathbb{A}\!\T_\ff \otimes_{{}_\mathbb{A}\!\HH} {}_\mathbb{A}\!\T_\g^*\cong\bigoplus_{i,j}V_i\otimes M_i\otimes_{{}_\mathbb{A}\!\HH}M_j^*\otimes U_j^*\cong\bigoplus_{i}V_i\otimes M_i\otimes_{{}_\mathbb{A}\!\HH}M_i^*\otimes U_i^*=\bigoplus_{i}V_i\otimes U_i^*,$$
which gives a multiplicity-free decomposition of ${}_\mathbb{A}\!\Tfg$ as an $({}_\mathbb{A}\!\Sc_\ff,{}_\mathbb{A}\!\Sc_\g)$-module.

Note that ${}_\mathbb{A}\!\Phi({}_\mathbb{A}\!\Sc_\ff)\cong\bigoplus_i\mathrm{End}_{\mathbb{A}}(V_i)$ by the Wedderburn-Artin Theorem since ${}_\mathbb{A}\!\Phi({}_\mathbb{A}\!\Sc_\ff)$ is semisimple.
We compute
$$\End_{{}_\mathbb{A}\!\Sc_\g}({}_\mathbb{A}\!\Tfg)\cong\End_{{}_\mathbb{A}\!\Sc_\g}(\bigoplus_{i}V_i\otimes U_i^*)\cong\bigoplus_i \mathrm{End}_{\mathbb{A}}(V_i)\otimes\mathrm{id}_{U_i^*}\cong\bigoplus_i \mathrm{End}_{\mathbb{A}}(V_i)\cong{}_\mathbb{A}\!\Phi({}_\mathbb{A}\!\Sc_\ff)$$
where the second ``$\cong$'' follows from Schur's Lemma. Similarly, $\End_{{}_\mathbb{A}\!\Sc_\ff}({}_\mathbb{A}\!\Tfg)={}_\mathbb{A}\!\Psi({}_\mathbb{A}\!\Sc_\g)$.
Thus the double centralizer property stated in the theorem is derived.
\end{proof}

Denote
$$\mathcal{M}_\ff=\mbox{the set of minimal parabolic subgroups in } \{P_\mu~|~\mu\in \Lambda_\ff\},$$
$$\mathcal{M}_\g=\mbox{the set of minimal parabolic subgroups in } \{P_\mu~|~\mu\in \Lambda_\g\}.$$
Clearly, the set $X_\ff$ contains at least one regular $W$-orbit if and only if $\mathcal{M}_\ff=\{B\}$.

Now we are ready to provide the following $(\Sc_\ff,\Sc_\g)$-duality over $\A=\mathbb{Z}[q,q^{-1}]$.
\begin{thm}
	\label{duality over Z[q,q{-1}]}
	If $\mathcal M_\ff=\mathcal M_\g$, then the actions
	$$\Sc_\ff \quad \stackrel{\Phi}\curvearrowright \quad \Tfg \quad
\stackrel{\Psi}\curvearrowleft \quad \Sc_\g$$
	satisfy
	$$\Sc_\ff \cong \Phi(\Sc_\ff)= \End_{\Sc_\g}(\Tfg), \quad  \End_{\Sc_\ff}(\Tfg)=
\Psi(\Sc_\g)\cong(\Sc_\g)^{op}.$$
\end{thm}
\begin{proof}
Take arbitrary nonzero $\chi=\sum_{\xi\in\Xi_\ff}a_\xi\chi_\xi \in \Sc_\ff$. There exists $(\gamma,w,\nu)\in \Xi_\ff$ such that $a_{(\gamma,w,\nu)}\neq0$. Since $\mathcal M_\ff=\mathcal M_\g$,  we can find $\mu \in \sF_\g$ such that $P_\mu\subset P_\nu$. Write $\Phi(\chi)(\chi_{(\nu,1,\mu)})=\chi*\chi_{(\nu,1,\mu)}=\sum_{\xi'\in\Xi_{\ff\g}}b_{\xi'}\chi_{\xi'}\in\Tfg$.
Corollary~\ref{coefficient} tells us that $b_{(\gamma, w, \mu)}=a_{(\gamma,w,\nu)}\neq0$, which implies $\Phi(\chi)\neq0$ and hence $\Phi$ is injective. So we have $\Sc_\ff \cong \Phi(\Sc_\ff)$.

It is obvious that  $\Phi(\Sc_\ff)\subset\End_{\Sc_\g}(\Tfg)$ by the definition of the convolution product.
Below we shall prove $\Phi(\Sc_\ff)\supset \End_{\Sc_\g}(\Tfg)$.

We regard $\mathrm{End}_\A(\Tfg)$ as a subring of $\mathrm{End}_{\mathbb{A}}({}_\mathbb{A}\!\Tfg)$ by the natural way, and hence $\End_{\Sc_\g}(\Tfg) \subset		\End_{{}_\mathbb{A}\!\Sc_\g}({}_\mathbb{A}\!\Tfg)={}_\mathbb{A}\!\Phi({}_\mathbb{A}\!\Sc_\ff)$
by Theorem~\ref{duality}.
	
Let us specialize $\mathbb{A}=\mathbb{Q}(q)$ in the following arguments. Take any $\tau \in \End_{\Sc_\g}(\Tfg)\subset {}_\mathbb{A}\!\Phi({}_\mathbb{A}\!\Sc_\ff)$. Under the assumption $\mathbb{A}=\mathbb{Q}(q)$, we can write $\tau=a^{-1}\tau'$ with $0\neq a\in\A$ and $\tau'\in\Phi(\Sc_\ff)$. Suppose $a\tau=\tau'=\Phi(\chi)$ with $\chi=\sum_{\xi\in\Xi_\ff}a_\xi\chi_\xi \in\Sc_\ff$. For any $(\gamma,w,\nu)\in \Xi_\ff$, we can find $\mu \in \sF_\g$ such that $P_\mu\subset P_\nu$ since $\mathcal M_\ff=\mathcal M_\g$. Compute $\chi\ast\chi_{(\nu,1,\mu)}=\Phi(\chi)(\chi_{(\nu,1,\mu)})=a\tau(\chi_{(\nu,1,\mu)}) =\sum_{\xi'\in\Xi_{\ff\g}}b_{\xi'}\chi_{\xi'} \in a\Tfg$. Therefore, $a_{(\gamma,w,\nu)}=b_{(\gamma, w, \mu)} \in a\A$ by Corollary~\ref{coefficient} again. Hence $a\tau=\Phi(\chi) \in a\Phi(\Sc_\ff)$, i.e., $\tau \in \Phi(\Sc_\ff)$, which implies $\End_{\Sc_\g}(\Tfg)\subset\Phi(\Sc_\ff)$. So $\Phi(\Sc_\ff)=\End_{\Sc_\g}(\Tfg)$ as desired.
	
Likewise $(\Sc_\g)^{op} \cong \Psi(\Sc_\g)=\End_{\Sc_\ff}(\Tfg)$.
\end{proof}

\begin{rem}
  If we take $X_\mathbf{g}$ to be a single regular $W$-orbit, the associated quantum
  Schur algebra is isomorphic to the Hecke algebra $\mathcal{H}$. So our geometric Howe
  dualities implies the quantum Schur dualities (of any finite type).
\end{rem}

\begin{rem}
For classical types ABCD and a special $W$-invariant finite subset $X_\ff$, it is known
that the convolution product on $\Sc_\ff$ admits a stabilization property, which brings
us a geometric realization of the quantum group $\qU$ and its coideal subalgebras (cf.
\cite{BLM90,BKLW18,FL15,LL21}). This geometric approach is also valid for affine type
(cf. \cite{Lu99,FLLLW20}). We will treat classical types in detail in latter sections,
where the Howe dualities are stated via quantum groups or $\imath$quantum groups instead
of quantum Schur algebras.
\end{rem}


\section{Howe duality for quantum general linear groups}\label{section3}

\subsection{Weights and orbits of type $\mA_{d-1}$}

Take $X=\sum_{i=1}^d \mathbb{Z}\delta_i$ to be the weight lattice for $\mathrm{GL}_d$,
where $\{\delta_i\}_{i=1}^d$ forms its standard basis. The Weyl group $\mathfrak{S}_d$
acts on $X$ by permutating $\delta_i$. For any positive integer $n \geq d$, we set
$$X_n=\{\sum_{i=1}^d a_i\delta_i  ~|~  a_i \in \mathbb{Z}, 1 \leq a_i \leq n, \forall 1
\leq i \leq d\},$$
which is clearly an $\mathfrak{S}_d$-invariant finite subset of $X$. We sometimes write a
weight by $(a_1,a_2,\ldots,a_d)$ instead of $\sum_{i=1}^d a_i\delta_i$.

Each $\mathfrak{S}_d$-orbit in $X_n$ can be described by the set of all compositions of
$d$ into $n$ parts
\begin{equation*}
\Lambda_{n,d}=\{\gamma=(\gamma_1,\gamma_2,\ldots,\gamma_n)  ~|~  \sum_{i=1}^n
\gamma_i=d\},
\end{equation*}
where an orbit $\gamma=(\gamma_1,\gamma_2,\ldots,\gamma_n) \in \Lambda_{n,d}$ consists of
all weights $\sum_{i=1}^d a_i\delta_i \in X_n$ such that $$\gamma_k=\sharp\{i ~|~
a_i=k,i=1,\ldots,d\},\quad(k=1,2,\ldots,n).$$
The unique anti-dominant element in an orbit $\gamma=(\gamma_1,\gamma_2,\ldots,\gamma_n)$
is
$$\mi_\gamma=(\underbrace{1,\ldots,1}_{\gamma_1},\underbrace{2,\ldots,2}_{\gamma_2},\ldots,\underbrace{n,\ldots,n}_{\gamma_n}).$$
Furthermore, the set $J_\gamma$ defined in \eqref{eq:PP} is
$$J_\gamma=\{1,2,\ldots,d\}\backslash\{\gamma_1,\gamma_1+\gamma_2,\ldots,\gamma_1+\ldots+\gamma_n\}.$$

\subsection{Flag varieties of type A}
We denote the set of all $n$-step partial flags of $\F^d$ by
$$\sF_{n,d}=\{\mathfrak{f}=(0=V_0\subset V_1 \subset \ldots \subset V_n=\F^d)\}.$$

We fix a basis $\{v_1,\ldots,v_d\}$ of $\F^d$, and set $W_i=\langle
v_1,\ldots,v_i\rangle$. For $\gamma =(\gamma_1,\gamma_2,\ldots,\gamma_n)\in \Lambda_{n,d}$, the parabolic subgroup $P_\gamma$, defined in Subsection~\ref{flv}, consists of the elements which stabilize the flag
$$\mathfrak{f}_\gamma:=(0 \subset W_{\gamma_1} \subset W_{\gamma_1+\gamma_2} \subset
\ldots \subset W_d=\F^d).$$
\begin{lem}
	\label{variety isomorphism}
	As varieties,
	$$\bigsqcup_{\gamma \in \Lambda_{n,d}} \mathrm{GL}_d/P_{\gamma} \simeq \sF_{n,d},
\quad [g] \mapsto g\mathfrak{f}_\gamma.$$
	\end{lem}

\begin{proof}
	It is known that for an algebraic group $G$ and a $G$-variety $X$, the orbit
containing $x \in X$ is isomorphic to $G/\mathrm{Stab}(x)$. So the map is well-defined and injective because of the fact that $\mathrm{Stab}(\mathfrak{f}_\gamma)= P_\gamma$.
	
	Let $\mathfrak{f}=(0=V_0\subset V_1 \subset \ldots \subset V_n=\F^d)$. Then there
exists $\gamma \in \Lambda_{n,d}$ and a basis $\{v_1', \ldots, v_d'\}$ such that
$\gamma_i= \dim V_i/V_{i-1}$ and $V_i=\langle v_1',\ldots,v_{\gamma_1+ \ldots
+\gamma_i}'\rangle$. Let $g \in \mathrm{GL}_d$ such that $g(v_i)=v_i'$, then $g
\mathfrak{f}_\gamma=\mathfrak{f}$. So this map is bijective.
\end{proof}

For any $m,n \in \N$, let $\mathrm{GL}_d$ act diagonally on the products $\sF_{m,d}
\times \sF_{n,d}$. It can be checked that the bijection shown in \cite{BLM90}
between $\mathrm{GL}_d \backslash \sF_{n,d} \times \sF_{n,d}$ and
$$\Theta_{n,d}=\{(a_{ij}) \in \mathrm{Mat}_{n\times n}(\mathbb{N}) ~|~ \sum_{1 \leq i \leq n,1 \leq
j \leq n} a_{ij}=d\},$$ induces a bijection between $\mathrm{GL}_d \backslash \sF_{m,d}
\times \sF_{n,d}$ and
$$\Theta_{m|n,d}=\{(a_{ij}) \in \mathrm{Mat}_{m \times n}(\mathbb{N}) ~|~ \sum_{1 \leq i
\leq m,1 \leq j \leq n} a_{ij}=d\}.$$
Denote
$$\T_{m|n,d}=\A_{\mathrm{GL}_d}(\sF_{m,d}
\times \sF_{n,d}), \quad \T_{m|n}=\bigoplus_{d=0}^{\infty} \T_{m|n,d},\quad\Sc_{n,d}=\T_{n|n,d}.$$
Here $\Sc_{n,d}$ (with its convolution product) is just the original quantum Schur
algebra introduced by Dipper and James \cite{DJ89}.

\subsection{Explicit action}

For any $A \in \Theta_{n,d}$ (resp. $\Theta_{m,d}$ and $\Theta_{m|n,d}$), let $\chi_A \in
\Sc_{n,d}$ (resp. $\Sc_{m,d}$ and $\T_{m|n,d}$) be the characteristic function of the
$\mathrm{GL}_d$-orbit in $\sF_{n,d} \times \sF_{n,d}$ (resp. $\sF_{m,d} \times \sF_{m,d}$
and $\sF_{m,d} \times \sF_{n,d}$) associated with $A=(a_{ij})$, and set $$[A]:=q^{\sum_{i
\geq k,j<l}a_{ij}a_{kl}} \chi_A.$$
Moreover, for $A\in \Theta_{m|n,d}$, let
\begin{align*}
\mathrm{row}(A)=
(\sum_{j=1}^n a_{1j},\sum_{j=1}^n a_{2j},\ldots,\sum_{j=1}^n a_{mj}),\quad
\mathrm{col}(A)=
(\sum_{i=1}^m a_{i1},\sum_{i=1}^m a_{i2},\ldots,\sum_{i=1}^m a_{in}).
\end{align*}
Similarly, we can define $\mathrm{row}(A)$ and $\mathrm{col}(A)$ for $A\in\Theta_{m,d}$
or $\Theta_{n,d}$.

For any $n \in \mathbb{N}$, denote $[n]=\frac{q^n-q^{-n}}{q-q^{-1}}$ the quantum
integer. Let $E_{ij}\in \mathrm{Mat}_{m\times m}(\mathbb{N})$ or $\mathrm{Mat}_{n\times
n}(\mathbb{N})$ be the matrix whose $(i,j)$-th entry is $1$ and others are $0$.
The following proposition can be obtained by a similar computation to
\cite[Lemma~3.4]{BLM90}.
\begin{prop} \label{action1}
	Let $A=(a_{ij}) \in \Theta_{m|n,d}$.
	\begin{itemize}
		\item[(1)] Assume $B,C \in \Theta_{m,d}$ such that $B-E_{i,i+1}$ and
$C-E_{i+1,i}$ are diagonal. If $\co(B)=\co(C)=\ro(A)$, then
		\begin{align*}
		[B] \cdot [A]&=\sum_{1 \leq j \leq n;a_{i+1,j}>0} q^{\sum_{k>j}
(a_{i+1,k}-a_{ik})} [a_{ij}+1][A+E_{ij}-E_{i+1,j}],\\
		[C] \cdot [A]&=\sum_{1 \leq j \leq n;a_{ij}>0} q^{\sum_{k<j}
(a_{ik}-a_{i+1,k})} [a_{i+1,j}+1][A+E_{i+1,j}-E_{ij}].
		\end{align*}
		\item[(2)] Assume $B,C \in \Theta_{n,d}$ such that $B-E_{i,i+1}$ and
$C-E_{i+1,i}$ are diagonal. If $\ro(B)=\ro(C)=\co(A)$, then
		\begin{align*}
		[A] \cdot [B]&=\sum_{1 \leq j \leq m;a_{ji}>0} q^{\sum_{k<j}
(a_{ki}-a_{k,i+1})} [a_{j,i+1}+1][A+E_{j,i+1}-E_{ji}],\\
		[A] \cdot [C]&=\sum_{1 \leq j \leq m;a_{j,i+1}>0} q^{\sum_{k>j}
(a_{k,i+1}-a_{ki})} [a_{ji}+1][A+E_{ji}-E_{j,i+1}].
		\end{align*}
	\end{itemize}
\end{prop}
\subsection{Quantum general linear groups}\label{qg}
Let $\qU_n$ denote the quantum group $U_q(\mathfrak{gl}_n)$ of type $\mathrm{A}_{n-1}$
over $\mathbb{A}$ with generators $E_i,F_i (i=1,2,\ldots,n-1)$ and $D_a^{\pm 1} (a=1,2,\ldots,n)$, subject to the following relations:
\begin{align*}
  D_aD_a^{-1}=1,\quad &D_{a}^{\pm1}D_{b}^{\pm1}=D_{b}^{\pm1}D_{a}^{\pm1},\\
  D_aE_iD_a^{-1}=q^{\delta_{ai}-\delta_{a+1,i}}E_i,\quad &D_aF_iD_1^{-1}=q^{\delta_{a+1,i}-\delta_{ai}}F_i,\\
  E_iF_j-F_jE_i=&\delta_{ij}\frac{K_i-K_i^{-1}}{q-q^{-1}},\quad\mbox{where $K_i=D_{i}D_{i+1}^{-1}$},\\
  E_iE_j=E_jE_i,\quad &F_iF_j=F_jF_i, \quad(|i-j|>1),\\
  E_i^2E_j+E_jE_i^{2}=(q+q^{-1})E_iE_jE_i, \quad &F_i^2F_j+F_jF_i^{2}=(q+q^{-1})F_iF_jF_i,\quad(|i-j|=1),
\end{align*}

There is a Hopf algebra structure on $\qU_n$ with the comultiplication $\Delta$, the
counit $\ve$, and the antipode $S$ as follows:
\begin{align*}
\Delta(E_i)&=E_i \otimes K_i^{-1} +1 \otimes E_i,\quad
\Delta(F_i)=F_i \otimes 1+K_i \otimes F_i,\quad
\Delta(D_a)=D_a \otimes D_a;\\
\ve(E_i)&=\ve(F_i)=0,\quad \ve(D_a)=1;\\
S(E_i)&=-K_i^{-1}E_i,\quad
S(F_i)=-F_iK_i,\quad
S(D_a)=D_a^{-1}.
\end{align*}

For $\lambda \in \Z^n$, a left (resp. right) $\qU_n$-module $M$ is called a highest
weight module with highest weight $\lambda$ if there exists a nonzero $v_\lambda \in M$
such that
\begin{align*}
E_i v_\lambda &=0,  (\forall 1\leq i<n),\quad
D_j v_\lambda =q^{\lambda_j} v_\lambda,  (\forall 1 \leq j \leq n),\quad
M =\qU_n v_\lambda \\
(\mbox{resp.}\quad
v_\lambda F_i &=0,  (\forall 1 \leq i<n),\quad
v_\lambda D_j =q^{\lambda_j} v_\lambda, (\forall 1 \leq j \leq n),\quad
M =v_\lambda \qU_n).
\end{align*}
The unique irreducible left (resp. right) module with highest weight $\lambda$ is denoted
by $L_\lambda^{[n]}$ (resp. $\widetilde{L}_\lambda^{[n]}$).

\subsection{Geometric construction}
Denote
\begin{align*}
\Theta_{n,d}^\diag=\{A \in \Theta_{n,d} ~|~ \mbox{$A$ is diagonal}\},\quad
\widetilde{\Theta}_n=\{(a_{ij}) \in \mathrm{Mat}_{n \times n}(\mathbb{Z}) ~|~ a_{ij} \geq
0, \forall i \neq j\}.
\end{align*}
Let $\widehat{\mathcal K}$ be the $\mathbb{A}$-space of all formal (possibly infinite)
$\mathbb{A}$-linear combinations $\sum_{A \in \widetilde{\Theta}_n} \kappa_A [A]$ with
certain finite conditions (see \cite[\S5.1]{BLM90}). The convolution product on
$\Sc_{n,d}$ can be lifted to $\widehat{\mathcal K}$ thanks to the stabilization property
shown in \cite[\S4]{BLM90}. Thus $\widehat{\mathcal K}$ is also an associative
$\mathbb{A}$-algebra. Moreover, it was verified in \cite[\S5.4]{BLM90} that there is an
embedding $\qU_n\hookrightarrow\widehat{\mathcal K}$, which induces a surjective
$\mathbb{A}$-algebra homomorphism $\kappa_{n,d}: \qU_n \twoheadrightarrow
{}_\mathbb{A}\!\Sc_{n,d}$ satisfying
\begin{align}\label{homo:k}
E_i &\mapsto \sum_{Z \in \Theta_{n,d-1}^\diag} [E_{i,i+1}+Z], \quad F_i \mapsto \sum_{Z
\in \Theta_{n,d-1}^\diag} [E_{i+1,i}+Z], \quad (1 \leq i<n), \\\nonumber
D_j &\mapsto \sum_{Z \in \Theta_{n,d}^\diag} q^{z_{jj}}[Z], \quad (1 \leq j \leq n).
\end{align}
The surjective $\mathbb{A}$-algebra homomorphism $\kappa_{m,d}$ (resp. $\kappa_{n,d}$)
implies a left $\qU_m$-action ${}_\mathbb{A}\Phi\circ\kappa_{m,d}$ (resp. right
$\qU_n$-action ${}_\mathbb{A}\Psi\circ\kappa_{n,d}$) on ${}_\mathbb{A}\!\T_{m|n,d}$,
where ${}_\mathbb{A}\Phi$ (resp. ${}_\mathbb{A}\Psi$) denotes the left
${}_\mathbb{A}\!\Sc_{m,d}$-action (resp. right ${}_\mathbb{A}\!\Sc_{n,d}$-action) on
${}_\mathbb{A}\!\T_{m|n,d}$. Thus we have the following double centralizer property for
$\qU_m$ and $\qU_n$ by Theorem~\ref{duality}.
\begin{thm}
	\label{thm:howeA}
	The actions
	$$\qU_m \quad \stackrel{\kappa_{m,d}}\twoheadrightarrow \quad
{}_\mathbb{A}\!\Sc_{m,d} \quad \stackrel{{}_\mathbb{A}\Phi}\curvearrowright \quad
{}_\mathbb{A}\!\T_{m|n,d} \quad \stackrel{{}_\mathbb{A}\Psi}\curvearrowleft \quad
{}_\mathbb{A}\!\Sc_{n,d} \quad \stackrel{\kappa_{n,d}}\twoheadleftarrow \quad \qU_n$$
	satisfy
	$$ {}_\mathbb{A}\Phi\circ\kappa_{m,d}(\qU_m)\cong
\End_{\qU_n}({}_\mathbb{A}\!\T_{m|n,d}), \quad
\End_{\qU_m}({}_\mathbb{A}\!\T_{m|n,d})\cong {}_\mathbb{A}\Psi\circ\kappa_{n,d}(\qU_n).$$
\end{thm}

\begin{rem}
	The above construction was firstly achieved in \cite{Ba07} (see also \cite{W01} for a non-quantized version).
\end{rem}

\subsection{Quantum coordinate algebra}\label{qca}

Let
$$\qU_n^\circ:=\{f \in \qU_n^* ~|~ \mbox{$\Ker f$ contains a cofinite ideal of
$\qU_n$}\}$$
denote the cofinite dual of $\qU_n$, which is equipped with a Hopf algebra structure
induced by the one of $\qU_n$.
Let $\mathbb{A}^n$ be the natural representation of $\qU_n$ with a standard basis $\{v_a
~|~ 1 \leq i \leq n\}$. That is,
\begin{equation}\label{naturalmoduleA}
E_i v_{k+1}=\delta_{ik} v_k, \quad F_i v_k=\delta_{ik} v_{k+1}, \quad D_j
v_k=q^{\delta_{jk}} v_k.
\end{equation}
Denote by $t_{ij} \in \qU_n^*$ $(1 \leq i,j \leq n)$, the matrix coefficients of the
$\qU_n$-module $\mathbb{A}^n$ relative to the above standard basis, i.e.,
$$xv_j=\sum_{i} v_i \langle t_{ij},x\rangle, \quad (\forall x \in \qU_n),$$
where $\langle \cdot,\cdot \rangle$ is the dual space pairing. Clearly, $t_{ij} \in
\qU_n^\circ$.
It is obvious by definition that
\begin{align}\label{tije}
\langle t_{jk},E_i\rangle=&\left\{\begin{array}{ll}
1, & \mbox{if $i=j=k-1$}\\
0, & \mbox{otherwise}
\end{array}\right.,
\quad
\langle t_{jk},F_i\rangle=\left\{\begin{array}{ll}
1, & \mbox{if $i=k=j-1$}\\
0, & \mbox{otherwise}
\end{array}\right.,\\ \nonumber
\langle t_{jk},D_i\rangle=&\left\{\begin{array}{ll}
q, & \mbox{if $i=j=k$}\\
1, & \mbox{if $i\neq j=k$}\\
0, & \mbox{otherwise}.
\end{array}\right.
\end{align}
The Hopf algebra structure on $\qU_n^\circ$ implies that the product $t_{i_1 j_1} \ldots
t_{i_d j_d}$ is the matrix coefficient of $(\mathbb{A}^n)^{\otimes d}$ such that
$$x(v_{j_1} \otimes \ldots \otimes v_{j_d})=\sum_{i_1,\ldots,i_d}v_{i_1} \otimes \ldots
\otimes v_{i_d}\langle t_{i_1 j_1} \ldots t_{i_d j_d},x\rangle, \quad(\forall
x\in\qU_n),$$ and that the comultiplication $\Delta^\circ$ of $\qU_n^\circ$ satisfies
\begin{equation}\label{comultioft}
\Delta^\circ(t_{ij})=\sum_k t_{ik} \otimes t_{kj}.
\end{equation}
Let $\mathcal{T}_n$ be the subbialgebra of $\qU_n^\circ$ generated by $t_{ij}$ $(1\leq
i,j\leq n)$, which is called the \emph{quantum coordinate algebra} of $\qU_n$.
Thanks to the Schur-Jimbo duality (i.e. the double centralizer property between
${}_\mathbb{A}\!\HH(\mathfrak{S}_d)$ and $\qU_n$ on $(\mathbb{A}^n)^{\otimes d}$), we can
obtain that, for $1\leq i<k\leq n,1\leq j<l \leq n$,
\begin{align}\label{tij}
t_{ij} t_{kj}=qt_{kj} t_{ij},\quad
t_{ij} t_{il}=qt_{il} t_{ij},\quad
t_{il} t_{kj}=t_{kj} t_{il},\quad
t_{ij} t_{kl}=t_{kl} t_{ij}+(q-q^{-1})t_{il} t_{kj}.
\end{align}

We shall use the lexicographical order $<$ on $\mathbb{Z}^2$, i.e.,
$$(i,j)<(k,l) \quad\Leftrightarrow\quad i<k \mbox{ or } i=k, j<l.$$

Denote $$\Theta_{n}=\mathrm{Mat}_{n\times n}(\mathbb{N})=\bigsqcup_{d=0}^\infty \Theta_{n,d}.$$
For $A \in \Theta_n$, set $t^{(A)}=\prod_{1\leq i,j\leq n}^{<}(t_{ij})^{a_{ij}}$, where the product is arranged in the way that $t_{ij}$ is positioned in front of $t_{kl}$ if $(i,j)<(k,l)$.
There is another order $<'$ on $\mathbb{Z}^2$ as follows:
$$(i,j)<'(k,l) \quad\Leftrightarrow\quad j<l \mbox{ or } j=l, i<k,$$
by which we can also set $t'^{(A)}=\prod_{1 \leq i,j \leq n}^{<'} (t_{ij})^{a_{ij}}$ for any $A \in \Theta_n$ in a similar way. The following lemma will be used to derive the last two formulas in Proposition~\ref{actiononta}.
\begin{lem}\label{t=tprime}
	For any $A \in \Theta_n$, we have $t^{(A)}=t'^{(A)}$.
\end{lem}
\begin{proof}
	If ``$i<k$ and $j>l$'' or ``$i>k$ and $j<l$'', it always holds that $t_{ij}t_{kl}=t_{kl}t_{ij}$ by the third equation in \eqref{tij}; otherwise, we have that $(i,j) \leq (k,l)$ if and only if $(i,j) \leq' (k,l)$. Thus $t^{(A)}=t'^{(A)}$.
\end{proof}

It is known (cf. \cite{Z02}) that $\{t^{(A)} ~|~ A \in \Theta_n\}$ is an
$\mathbb{A}$-basis of $\mathcal{T}_n$ (called \emph{monomial basis}), and that $\mathcal{T}_n$ forms
a left module algebra and a right module algebra \footnote{For a bialgebra $B$, an algebra $A$ is called a left (resp. right) module algebra over $B$ if
	\begin{itemize}
		\item $A$ is a left (resp. right) $B$-module, and
		\item the multiplication of $A$ is a $B$-module homomorphism from $A \otimes A$ to
		$A$.
\end{itemize}} over $\qU_n$ by, respectively, the following left action and right
action:

\begin{align}\label{actionont}
x \cdot f:=\sum_{(f)} f_{(1)}\langle f_{(2)},x\rangle,\qquad
f \cdot x:=\sum_{(f)} \langle f_{(1)},x\rangle f_{(2)},
\end{align}
where $x \in \qU_n, f \in \mathcal{T}_n$ and $\Delta^\circ(f)=\sum_{(f)} f_{(1)} \otimes
f_{(2)}$.

The $\qU_n$-action on $\mathcal{T}_n$ are formulated
explicitly  in the following proposition.
\begin{prop}\label{actiononta}
	Let $A \in \Theta_n$. For $E_i, F_i\in \qU_n$, we have
	\begin{align*}
	E_i \cdot t^{(A)}&=\sum_{1 \leq j \leq n;a_{j,i+1}>0} q^{\sum_{k>j}
		(a_{k,i+1}-a_{ki})} [a_{j,i+1}]t^{(A+E_{ji}-E_{j,i+1})}, \\
	F_i \cdot t^{(A)}&=\sum_{1 \leq j \leq n;a_{ji}>0} q^{\sum_{k<j} (a_{ki}-a_{k,i+1})}
	[a_{ji}]t^{(A+E_{j,i+1}-E_{ji})};
	\end{align*}
	and for $E_i, F_i\in \qU_m$, we have
	\begin{align*}
	t^{(A)} \cdot E_i&=\sum_{1 \leq j \leq n;a_{ij}>0} q^{\sum_{k \geq j}
		(a_{i+1,k}-a_{ik})+1} [a_{ij}] t^{(A+E_{i+1,j}-E_{ij})},\\
	t^{(A)} \cdot F_i&=\sum_{1 \leq j \leq n;a_{i+1,j}>0} q^{\sum_{k \leq j}
		(a_{ik}-a_{i+1,k})+1} [a_{i+1,j}] t^{(A+E_{ij}-E_{i+1,j})}.
	\end{align*}
\end{prop}
\begin{proof}
	By \eqref{tije}, \eqref{comultioft} and \eqref{actionont}, we have
	\begin{align*}
	E_i \cdot t_{j,k+1}=\sum_{l} t_{jl}\langle t_{l,k+1},E_i\rangle=\delta_{ik}t_{jk}.
	\end{align*}
	Similarly, we can compute that
	$$F_i \cdot t_{jk}=\delta_{ik}t_{j,k+1},
	\quad D_i \cdot t_{jk}=q^{\delta_{ik}}t_{jk},$$
	$$t_{jk} \cdot E_i=\delta_{ij}t_{j+1,k},\quad
	t_{j+1,k} \cdot F_i=\delta_{ij}t_{jk},\quad
	t_{jk} \cdot D_i=q^{\delta_{ij}}t_{jk}.$$
	Thus using the comultiplication $\Delta(E_i)=E_i \otimes K_i^{-1} +1 \otimes E_i$, we
	have
	\begin{equation*}
	E_i \cdot t_{j,k+1}^d=\delta_{ik}\sum_{c=1}^d t_{j,k+1}^{c-1}(E_i \cdot
	t_{j,k+1})(K_i^{-1} \cdot t_{j,k+1})^{d-c}=\delta_{ik}\sum_{c=1}^d
	q^{d-c}t_{j,k+1}^{c-1}t_{jk}t_{j,k+1}^{d-c}=[d]t_{jk}t_{j,k+1}^{d-1}
	\end{equation*} and hence
	\begin{align*}
	E_i \cdot t^{(A)}&=\sum_{1 \leq j \leq n} (\prod_{(k,\ell)<(j,i+1)}^{<}
	t_{k\ell}^{a_{k\ell}})(E_i \cdot t_{j,i+1}^{a_{j,i+1}})(K_i^{-1} \cdot
	(\prod_{(k,\ell)>(j,i+1)}^{<} t_{k\ell}^{a_{k\ell}}))\\
	&=\sum_{1 \leq j \leq n;a_{j,i+1}>0} q^{\sum_{k>j} (a_{k,i+1}-k_{ki})}
	[a_{j,i+1}]t^{(A+E_{ji}-E_{j,i+1})}.
	\end{align*}
	The other formulas can be computed similarly. We note that we need Lemma~\ref{t=tprime} to derive the formulas of right actions.
\end{proof}

\subsection{Multiplicity-free decomposition} Let $s=\max\{m,n\}$. Denote
$$\Theta_{m|n}=\mathrm{Mat}_{m\times n}(\mathbb{N})=\bigsqcup_{d=0}^\infty \Theta_{m|n,d},$$ which can be regarded as a subset of $\Theta_s$ by the natural way. In \cite{Z02}, Zhang constructed a
subalgebra $\mathcal{V}_{m|n}$ of $\mathcal{T}_s$ with an $\mathbb{A}$-basis
$\{t^{(A)} ~|~ A\in\Theta_{m|n}\}$.

\begin{thm}\cite[Theorem~1.1]{Z02} \label{thm:zhang}
	The subalgebra $\mathcal{V}_{m|n}$ forms a $(\qU_n, \qU_m)$-module algebra and admits the following
multiplicity-free decomposition:
	$$\mathcal{V}_{m|n} \cong \bigoplus_{\lambda \in \Par_{\min(m,n)}}L_{\lambda}^{[n]}
\otimes \widetilde{L}_\lambda^{[m]},$$
	where $\Par_{\min(m,n)}$ is the set of weights corresponding to partitions with at
most $\min(m,n)$ parts.
\end{thm}

For $A \in \Theta_{m|n}$,
denote
\begin{equation}\label{def:varA}
\langle A\rangle:=q^{\sum_{1\leq i\leq m} \frac{\ro_i(A)(\ro_i(A)+1)}{2}}
\frac{t^{(A)}}{\prod_{(i,j)}[a_{ij}]!} \in \mathcal{V}_{m|n},
\end{equation}
where $\ro_i(A)=\sum_{j=1}^n a_{ij}$.

The following corollary is clear by Proposition~\ref{actiononta} and \eqref{def:varA}.
\begin{cor} The set $\{\langle A\rangle ~|~ A\in\Theta_{m|n}\}$ forms an
$\mathbb{A}$-basis of $\mathcal{V}_{m|n}$. The explicit actions of $\qU_m$ and $\qU_n$ on
the basis elements are as follows: for $E_i,F_i\in\qU_n$,
	\begin{align*}
	E_i \cdot \langle A\rangle &= \sum_{1 \leq j \leq m;a_{j,i+1}>0} q^{\sum_{k>j}
(a_{k,i+1}-a_{ki})}[a_{ji}+1]\langle A+E_{ji}-E_{j,i+1}\rangle,\\
	F_i \cdot \langle A\rangle &= \sum_{1 \leq j \leq m;a_{ji}>0} q^{\sum_{k<j}
(a_{ki}-a_{k,i+1})}[a_{j,i+1}+1]\langle A+E_{j,i+1}-E_{ji}\rangle;
\end{align*}
and for $E_i,F_i\in\qU_m$,
\begin{align*}\langle A\rangle \cdot E_i &= \sum_{1 \leq j \leq n;a_{ij}>0} q^{\sum_{k<j}
(a_{ik}-a_{i+1,k})}[a_{i+1,j}+1]\langle A+E_{i+1,j}-E_{ij}\rangle,\\
	\langle A\rangle \cdot F_i &= \sum_{1 \leq j \leq n;a_{i+1,j}>0} q^{\sum_{k>j}
(a_{i+1,b}-a_{ik})}[a_{ij}+1]\langle A+E_{ij}-E_{i+1,j}\rangle.
	\end{align*}
\end{cor}

Let $\mathcal{V}_{m|n,d}$ be the subspace of $\mathcal{V}_{m|n}$ spanned by $\{t^{(A)}~|~A\in\Theta_{m|n,d}\}$.
Comparing the above corollary with Proposition~\ref{action1} together with the
homomorphism $\kappa_{n,d}$ in \eqref{homo:k}, we get the following isomorphism of
$(\qU_n, \qU_m)$-modules.
\begin{thm} \label{equivalent}
	There are isomorphisms between the $(\qU_n, \qU_m)$-modules:
	\begin{align*}
	\mathcal{V}_{m|n} \cong {}_\mathbb{A}\!\T_{n|m}, \quad \mathcal{V}_{m|n,d} \cong {}_\mathbb{A}\!\T_{n|m,d}:\quad \langle A\rangle \mapsto [A'],
	\end{align*}
	where $A'$ is the transposition of $A$.
\end{thm}
Thanks to Theorems~\ref{thm:zhang} \& \ref{equivalent}, we obtain the following result
which is a graded version of \cite[Theorem~1.1]{Z02}.
\begin{thm}
	As a $(\qU_m,\qU_n)$-module, we have
	$${}_\mathbb{A}\!\T_{m|n,d}= \bigoplus_{\lambda \in \Par_{\min(m,n)}(d)}
L_{\lambda}^{[m]} \otimes \widetilde{L}_\lambda^{[n]},$$
where $\Par_{\min(m,n)}$ is the set of weights corresponding to partitions of $d$ with at
most $\min(m,n)$ parts.
\end{thm}


\section{$\imath$Howe duality for $\imath$quantum groups of type
AIII/IV}\label{section4}
In this section, since a symmetric (resp. skew-symmetric) bilinear form on $\mathbb{F}^{2d+1}$ (resp. $\mathbb{F}^{2d}$) will be employed, we shall always assume $\mathrm{char}(\mathbb{F})\neq2$ for the finite field $\mathbb{F}$. This restriction is not essential and can be removed if we use a combinatorial approach (cf. \cite{LL21}) instead.
\subsection{Weights and orbits of type $\mB_d$}\label{weightB}
We fix $m,n\in\mathbb{N}$ and let
$$N=2n+1,\quad M=2m+1,\quad D=2d+1.$$
Let $\mathrm{SO}_D$ be the special orthogonal group whose natural module $\F^D$ is equipped with a
non-degenerate symmetric bilinear form $(\cdot,\cdot)$ satisfying $(v_i,v_j)=\delta_{i,-j}$ for a given basis $\{v_{-d},\ldots,v_{d}\}$. The weight lattice
for $\mathrm{SO}_D$ is
$X=X^0 \sqcup X^{\half}$ where
$$X^0= \sum_{i=1}^d \mathbb{Z} \delta_i, \qquad X^{\half}= \sum_{i=1}^d(\half+\mathbb{Z})
\delta_i.$$
The Weyl group $W_{B_d}= \mathfrak{S}_d \ltimes \mathbb{Z}_2^d$ acts on $X$ by
permutating $\delta_i$ and changing the signs of coefficients of $\delta_i$.
We take
\begin{align*}
X_n^0=\{\sum_{i=1}^d a_i\delta_i ~|~ a_i \in \mathbb{Z},|a_i| \leq n, \forall i\},\quad
X_n^{\half}=\{\sum_{i=1}^d a_i\delta_i ~|~ a_i \in \half+\mathbb{Z},|a_i|< n, \forall
i\}
\end{align*}
and denote
\begin{align*}
\Lambda_{n,d}^\jmath=&\{\gamma=(\gamma_{-n},\ldots,\gamma_{-1},2\gamma_0+1,\gamma_1,\ldots,\gamma_n)
~|~  \sum_{i=0}^n \gamma_i=d,\gamma_i=\gamma_{-i}\},\\
\Lambda_{n,d}^\imath=&\{\gamma=(\gamma_{-n},\ldots,\gamma_{-1},1,\gamma_1,\ldots,\gamma_n)\in\Lambda_{n,d}^\jmath\}\subset \Lambda_{n,d}^\jmath.
\end{align*}
Each $W_{B_d}$-orbit in $X_n^0$ can be indexed by the set $\Lambda_{n,d}^\jmath$, while each
$W_{B_d}$-orbit in $X_n^\half$ can be indexed by $\Lambda_{n,d}^\imath$. Precisely, an orbit
$\gamma\in \Lambda_{n,d}^\jmath$ consists of all
weights $\sum_{i=1}^d a_i\delta_i \in X_n^0$ such that
$$\gamma_k=\sharp \{i ~|~  |a_i|=k, i=1,\ldots,d\}, \quad (k=0,1,2,\ldots,n),$$
while an orbit
$\gamma \in \Lambda_{n,d}^\imath$ consists of all weights
$\sum_{i=1}^d a_i\delta_i \in X_n^\half$ such that
$$\gamma_k=\sharp \{i ~|~  |a_i|=k-\half, i=1,\ldots,d\}, \quad (k=1,2,\ldots,n).$$

\subsection{Flag varieties of type $\mB$}
Denote
\begin{align*}
\sF_{n,d}^{\mB,\jmath}=&\{\mathfrak{f}=(0=V_{-n-\half}\subset V_{-n+\half}\subset\cdots\subset
V_{n+\half}=\F^D) \in \sF_{N,D} ~|~ V_i=V_j^{\perp},\mbox{ if } i+j=0\},\\
\sF_{n,d}^{\mB,\imath}=&\{\mathfrak{f}\in\sF_{n,d}^{\mB,\jmath}~|~ \dim V_{-\frac{1}{2}}=\dim V_\frac{1}{2}-1\}\subset \sF_{n,d}^{\mB,\jmath}.
\end{align*}
Set $W_{i+\half}=\langle v_{-d},\ldots, v_i\rangle$. For
$\gamma\in\Lambda_{n,d}^\jmath$, now the parabolic subgroup $P_\gamma$ becomes the one
consisting of the elements which stabilize the flag
$$\mathfrak{f}_\gamma:=(0=W_{-d-\half}\subset
W_{-d+\gamma_{-n}-\half}\subset\cdots\subset W_{d-\gamma_n+\half}\subset
W_{d+\half}=\F^D).$$
\begin{lem}
	\label{variety isomorphism of type B}
	As varieties,
	$$\bigsqcup_{\gamma \in \Lambda_{n,d}^\mathfrak{b}} \mathrm{SO}_D/P_{\gamma} \simeq
\sF_{n,d}^{\mB,\mathfrak{b}},\ (\mathfrak{b}=\imath,\jmath):\quad [g] \in \mathrm{SO}_D/P_\gamma \mapsto g\mathfrak{f}_\gamma.$$
\end{lem}
\begin{proof}
	As the same as type A, $\mathrm{Stab}(\mathfrak{f}_\gamma)= P_\gamma$. So both of the maps (for type $\jmath$ and type $\imath$) are
well-defined and injective. Below we only need to show the surjectivity of the map for type $\jmath$.
	
	Let $\mathfrak{f}=(0=V_{-n-\half} \subset \ldots \subset V_{n+\half}=\F^d)$, $V$ a
maximal isotropic subspace containing $V_{-\half}$. We have $V \subset V_\half$ since
$V_\half=V_{-\half}^\perp$. Let $\gamma_i=\dim V_{i+\half}/V_{i-\half}$. Then there
exists a basis $\{v_{-d}',\ldots,v_0'\}$ of $V^\perp$, such that $(v_0',v_0')=1$,
$V=\langle v_{-d}',\ldots, v_{-1}'\rangle$ and $V_{i-\half}=\langle v_{-d}',\ldots,
v_{-d+\gamma_{-n}+\ldots+\gamma_i}'\rangle$ for $-n \leq i \leq 0$. Let $v_i' \in \F^D \
(1 \leq i \leq d)$ such that $(v_i',v_{-j}')=0$ if $j>i$ and $(v_i',v_{-i}')=1$. Clearly
such $v_i'$'s exist. Set
	\begin{align*}
	w_i=
	\begin{cases}
	v_i'-\sum_{j<i}(v_i',v_j')v_j', \quad &i\leq 0,\\
	v_i'-\sum_{j>i}(v_i',v_{-j}')v_{-j}'-\frac{(v_i',v_{-i}')}{2}, \quad &i>0.
	\end{cases}
	\end{align*}
	We have $(w_i,w_j)=\delta_{i,-j}$ and $V_{i+\half}=\langle w_{-d},\ldots,
w_{-d+\gamma_{-n}+\ldots+\gamma_i}\rangle$ for $-n \leq i<0$, hence $V_{i-\half}=\langle
w_{-d},\ldots, w_{d-\gamma_{-n}-\ldots-\gamma_i}\rangle$ for $0<i \leq n$. Let $g' \in
\mathrm{GL}_D$ such that $g'(v_i)=w_i$, and $g=\det(g')g$, then $g \in \mathrm{SO}_D$ and
$g\mathfrak{f}_\gamma=\mathfrak{f}$, so this map is bijective.
\end{proof}

Denote
\begin{align}\label{xind}
\Xi_{n,d}^\jmath&=\{(a_{ij})_{-n\leq i,j\leq n} \in \mathrm{Mat}_N(\mathbb{N}) ~|~ a_{ij}=a_{-i,-j}, \sum_{i,j} a_{ij}=D\},\\\nonumber
\Xi_{n,d}^\imath&=\{(a_{ij}) \in \Xi_{n,d}^\jmath ~|~ a_{0i}=a_{i0}=0 (i\neq0), a_{00}=1\},\\\nonumber
\Xi_{m|n,d}^\jmath&=\{(a_{ij})_{-m\leq i\leq m;-n\leq j\leq n} \in \mathrm{Mat}_{M\times N}(\mathbb{N}) ~|~
a_{ij}=a_{-i,-j}, \sum_{i,j} a_{ij}=D\},\\\nonumber
\Xi_{m|n,d}^{\jmath\imath}&=\{(a_{ij}) \in \Xi_{m|n,d}^\jmath ~|~
a_{i0}=0 (i\neq0), a_{00}=1\},\\\nonumber
\Xi_{m|n,d}^{\imath\jmath}&=\{(a_{ij}) \in \Xi_{m|n,d}^\jmath ~|~
a_{0i}=0 (i\neq0), a_{00}=1\},\\\nonumber
\Xi_{m|n,d}^\imath&=\{(a_{ij}) \in \Xi_{m|n,d}^\jmath ~|~
a_{i0}=a_{0j}=0 (i,j\neq0), a_{00}=1\}.
\end{align}
%
Let $\mathrm{SO}_D$ act diagonally on the products $\sF_{m,d}^{\mB,\mathfrak{b}}
\times \sF_{n,d}^{\mB,\mathfrak{c}}$, ($\mathfrak{b,c}\in\{\imath,\jmath\}$).
It has been shown in \cite[Lemma~2.1 \& Lemma~5.1]{BKLW18} that there is a bijection $$\mathrm{SO}_{D}
\backslash \sF_{n,d}^{\mB,\mathfrak{b}} \times \sF_{n,d}^{\mB,\mathfrak{b}} \leftrightarrow \Xi_{n,d}^{\mathfrak{b}}, \quad(\mathfrak{b}\in\{\imath,\jmath\}).$$ Moreover, a similar argument brings us the following bijection
\begin{equation*}
\mathrm{SO}_{D} \backslash \sF_{m,d}^{\mB,\mathfrak{b}} \times \sF_{n,d}^{\mB,\mathfrak{c}} \leftrightarrow \Xi_{m|n,d}^{\mathfrak{b,c}}, \quad (\mathfrak{b,c}\in\{\imath,\jmath\}).
\end{equation*}

We set
$$\T_{m|n,d}^{\mathfrak{bc}}=\A_{\mathrm{SO}_D}(\sF_{m,d}^{\mB,\mathfrak{b}} \times \sF_{n,d}^{\mB,\mathfrak{c}}), \quad
\T_{m|n}^{\mathfrak{bc}}=\bigoplus_{d=0}^{\infty} \T_{m|n,d}^{\mathfrak{bc}}, \quad
\Sc_{n,d}^\mathfrak{b}=\T_{n|n,d}^{\mathfrak{bb}},\quad(\mathfrak{b,c}\in\{\imath,\jmath\}).$$
Here $\Sc^\mathfrak{b}_{n,d}$ (together with its convolution product) is called an \emph{$\imath$Schur
algebra}.

For $\mathfrak{b,c}\in\{\imath,\jmath\}$ and $A \in \Xi_{n,d}^\mathfrak{b}$ (resp. $\Xi_{m,d}^{\mathfrak{b}}$ and $\Xi_{m|n,d}^{\mathfrak{bc}}$), let $\chi_A \in
\Sc_{n,d}^\mathfrak{b}$ (resp. $\Sc_{m,d}^\mathfrak{b}$ and $\T_{m|n,d}^\mathfrak{bc}$) be the
characteristic function of the $\mathrm{SO}_D$-orbit in $\sF_{n,d}^{\mB,\mathfrak{b}} \times
\sF_{n,d}^{\mB,\mathfrak{b}}$ (resp. $\sF_{m,d}^{\mB,\mathfrak{b}} \times \sF_{m,d}^{\mB,\mathfrak{b}}$ and $\sF_{m,d}^{\mB,\mathfrak{b}} \times
\sF_{n,d}^{\mB,\mathfrak{c}}$) associated with $A$. Denote
\begin{equation}\label{[A]b}
[A]=q^{\half(\sum_{i \geq k,j<l}a_{ij}a_{kl}-\sum_{i \geq 0,j<0} a_{ij})}\chi_A.
\end{equation}
We remark here that $\half(\sum_{i
\geq k,j<l}a_{ij}a_{kl}-\sum_{i \geq 0,j<0} a_{ij})$ is always an integer.

\subsection{Explicit action}
Let $E_{ij}^\theta=E_{ij}+E_{-i,-j}\in \mathrm{Mat}_{[-m,m]\times [-m,m]}(\mathbb{N})$ or
$\mathrm{Mat}_{[-n,n]\times [-n,n]}(\mathbb{N})$.
For $A=(a_{ij}) \in \Xi_{m|n,d}^{\mathfrak{bc}}$, let
\begin{equation}\label{asharp}
a_{ij}^\sharp=\left\{\begin{array}{ll}
\frac{a_{00}-1}{2}, & \mbox{if $(i,j)=(0,0)$};\\
a_{ij}, &\mbox{otherwise,}
\end{array}\right.
\end{equation}
and define
\begin{align*}
\ro(A)&=(\sum_{j=-n}^na_{-m,j},\sum_{j=-n}^na_{-m+1,j},\ldots,\sum_{j=-n}^na_{m,j}),\\
\co(A)&=(\sum_{i=-m}^ma_{i,-n},\sum_{i=-m}^ma_{i,-n+1},\ldots,\sum_{i=-m}^ma_{i,n}).
\end{align*} The definitions of $\ro(A)$ and $\co(A)$ for $A\in \Xi_{m,d}^\mathfrak{b}$ or $\Xi_{n,d}^\mathfrak{b}$
are similar.
We have the following formulas
about the left $\Sc_{m,d}^\mathfrak{b}$-action $\Phi$ and right $\Sc_{n,d}^\mathfrak{c}$-action
$\Psi$ on $\T_{m|n,d}^{\mathfrak{bc}}$.
\begin{prop}\label{prop:actionsj}
	Let $A=(a_{ij}) \in \Xi_{m|n,d}^\mathfrak{bc}$, $(\mathfrak{b,c}\in\{\imath,\jmath\})$.
	\begin{itemize}
		\item[(1)] Assume $B,C \in \Xi_{m,d}^\mathfrak{b}$ such that $B-E_{i,i+1}^\theta$ and
		$C-E_{i+1,i}^\theta$ are diagonal.
		If $\co(B)=\co(C)=\ro(A)$, then
		\begin{align*}
		[B] \cdot [A]&=\sum_{-n \leq j \leq n;a_{i+1,j}>0} q^{\sum_{k>j}
			(a_{i+1,k}-a_{ik})}{[a_{ij}+1]}[A+E_{ij}^\theta-E_{i+1,j}^\theta],\\
		[C] \cdot [A]&=\sum_{-n \leq j \leq 0;a_{ij}^\sharp>0} q^{\sum_{k<j}
			(a_{ik}-a_{i+1,k})}[a_{i+1,j}+1][A+E_{i+1,j}^\theta-E_{ij}^\theta]\\
		&+\sum_{0<j \leq n;a_{ij}>0} q^{\sum_{k<j}
			(a_{ik}-a_{i+1,k})-\delta_{0i}}[a_{i+1,j}+1][A+E_{i+1,j}^\theta-E_{ij}^\theta].
		\end{align*}
		\item[(2)] Suppose $B,C \in \Xi_{n,d}^\mathfrak{c}$ such that $B-E_{i,i+1}^\theta$ and
		$C-E_{i+1,i}^\theta$ are diagonal. If $\ro(B)=\ro(C)=\co(A)$, then
		\begin{align*}
		[A] \cdot [B]&=\sum_{-m \leq j \leq 0;a_{ji}^\sharp>0} q^{\sum_{k<j}
			(a_{ki}-a_{k,i+1})}[a_{j,i+1}+1][A+E_{j,i+1}^\theta-E_{ji}^\theta]\\
		&+\sum_{0<j \leq m;a_{ji}>0} q^{\sum_{k<j}
			(a_{ki}-a_{k,i+1})-\delta_{0i}}[a_{j,i+1}+1][A+E_{j,i+1}^\theta-E_{ji}^\theta],\\
		[A] \cdot [C]&=\sum_{-m \leq j \leq m;a_{j,i+1}>0} q^{\sum_{k>j}
			(a_{k,i+1}-a_{ki})}{[a_{ji}+1]}[A+E_{ji}^\theta-E_{j,i+1}^\theta].
		\end{align*}
\item[(3)] If $\mathfrak{b}=\imath$ and $D \in \Xi_{m,d}^{\imath}$ such that
$D-E_{1,-1}^\theta$ are diagonal and $\co(D)=\ro(A)$, then
	\begin{align*}
	[D] \cdot [A]&=(q^{\sum_{j \geq 0} a_{1j}-\sum_{j<0} a_{1j}}-q^{\sum_j a_{1j}})[A]\\
	&+\sum_{-n \leq j \leq n;a_{1j}>0} q^{\sum_{k>j} (a_{1j}-a_{-1,j})-a_{0j}+\sum_{k<0} \delta_{jk}} [a_{-1,j}+1-\delta_{0j}][A+E_{-1,j}^\theta-E_{1j}^\theta].
	\end{align*}
	If $\mathfrak{c}=\imath$ and $D \in \Xi_{n,d}^{\imath}$ such that $D-E_{1,-1}^\theta$ are diagonal and
$\ro(D)=\co(A)$, then
	\begin{align*}
	[A] \cdot [D]&=(q^{\sum_{j \geq 0} a_{j1}-\sum_{j<0} a_{j1}}-q^{\sum_j a_{j1}})[A]\\
	&+\sum_{-m \leq j \leq m;a_{j1}>0} q^{\sum_{k>j} (a_{j1}-a_{j,-1})-a_{j0}+\sum_{k<0} \delta_{jk}} [a_{j,-1}+1-\delta_{0j}][A+E_{j,-1}^\theta-E_{j1}^\theta].
	\end{align*}
	\end{itemize}
\end{prop}
\begin{proof}
 Items (1) and (2) can be derived by imitating the computation of \cite[Theorem 3.7]{BKLW18}. Then item (3) follows by $$[D]=[D-E_{1,-1}^\theta+E_{10}^\theta][D-E_{1,-1}^\theta+E_{01}^\theta]-q[d_{11}+1][D-E_{1,-1}^\theta+E_{11}^\theta],$$
 where $d_{11}$ is the $(1,1)$-th entry of $D$.
\end{proof}
%
\subsection{The $\imath$quantum group $\qU_n^\jmath$}
Denote
$$\mathbb{I}_N=\{-n,-n+1,\ldots,n-1,n\}, \quad
\mathbb{I}_N^{\half}=\{-n+\half,\ldots,n-\half\}.$$
Let $\qU_N$ denote the quantum group $\qU(\mathfrak{gl}_N)$ of type $\mathrm{A}_{2n}$
over $\mathbb{A}$ with generators $E_i, F_i, (i \in \mathbb{I}_N^{\half})$ and $D_j^{\pm
1}, (j \in \mathbb{I}_N)$ defined in Subsection~\ref{qg}. Here we not only replace $n$ by $N$ in the definition but also take a shift by $-n-\half$ (resp. $-n-1$) on the index of $E_i, F_i$ (resp. $D_j$). Write $K_i=D_{i-\half} D_{i+\half}^{-1}, (i \in
\mathbb{I}_N^{\half})$.

The \emph{$\imath$quantum group} $\qU_n^\jmath$ is the $\mathbb{A}$-subalgebra of $\qU_N$ generated by
\begin{align*}
e_i&=E_{i+\half}+K_{i+\half}^{-1}F_{-(i+\half)}, \quad
f_i=E_{-(i+\half)}+F_{i+\half}K_{-(i+\half)}^{-1}, \quad (0 \leq i<n),\\
d_0^{\pm 1}&=D_0^{\pm 1}, \quad d_j^{\pm 1}=(D_j D_{-j})^{\pm 1}, \quad (0<j \leq n).
\end{align*}
Write $k_i=K_{i+\half} K_{-i-\half}^{-1}$.

It is easy to check that
\begin{align*}
\Delta(e_i)&=e_i \otimes K_{i+\half}^{-1}+1 \otimes E_{i+\half}+k_i^{-1} \otimes
K_{i+\half}^{-1}F_{-(i+\half)},\\
\Delta(f_i)&=f_i \otimes K_{-(i+\half)}^{-1}+1 \otimes E_{-(i+\half)}+k_i \otimes
F_{i+\half}K_{-(i+\half)}^{-1},\\
\Delta(d_j)&=d_j \otimes d_j.
\end{align*} So $\Delta(\qU_n^\jmath) \subset \qU_n^\jmath \otimes \qU_N$, which means
$\qU_n^\jmath$ is a right coideal of $\qU_N$.
Its specialization at $q\to1$ is
$U(\mathfrak{gl}_{n+1}\oplus \mathfrak{gl}_n)$.

It has been shown in \cite[\S 4]{BKLW18} that Beilinson-Lusztig-MacPherson's realization
of general linear quantum groups still makes sense for $\qU_n^\jmath$. Such realization
induces a
surjective $\mathbb{A}$-algebra homomorphism $\kappa_{n,d}^\jmath: \qU_n^\jmath
\twoheadrightarrow {}_\mathbb{A}\!\Sc_{n,d}^\jmath$ satisfying
\begin{align}\label{homo:kj}
e_i& \mapsto \sum_{Z \in \Xi_{n,d-1}^\diag} [E_{i,i+1}^\theta+Z], \quad f_i \mapsto
\sum_{Z \in \Xi_{n,d-1}^\diag} [E_{i+1,i}^\theta+Z], \quad (0\leq i<n),\\\nonumber
d_0& \mapsto \sum_{Z \in \Xi_{n,d}^\diag} q^{z_{00}^\sharp}[Z], \quad d_j \mapsto
\sum_{Z \in \Xi_{n,d}^\diag} q^{z_{jj}}[Z], \quad (0 < j \leq n),
\end{align}
where
$\Xi_{n,d}^\diag=\{A \in \Xi_{n,d}^\jmath ~|~ \mbox{$A$ is diagonal}\}$ and the notation $z_{00}^\sharp$ is defined in \eqref{asharp}.
%

\subsection{The $\imath$quantum group $\qU_n^{\imath}$}

Denote
$$\mathbb{I}_{2n}=\{-n+\half,\ldots,n-\half\}, \quad
\mathbb{I}_{2n}^{\half}=\{-n+1,\ldots,n-1\}.$$
Let $\qU_{2n}$ denote the quantum group $\qU(\mathfrak{gl}_{2n})$ of type $\mA_{2n-1}$
with generators $E_i, F_i, (i \in \mathbb{I}_{2n}^{\half})$ and $D_j^{\pm 1}, (j \in
\mathbb{I}_{2n})$. Here we not only replace $n$ by $2n$ in the definition of $\qU_{n}$ in Subsection~\ref{qg} but also take a shift by $-n$ (resp. $-n-\half$) on the index of $E_i, F_i$ (resp. $D_j$). Denote $K_i=D_{i-\half} D_{i+\half}^{-1}$ and $k_i=K_i K_{-i}^{-1}, (i \in\mathbb{I}_{2n}^{\half})$.

The \emph{$\imath$quantum group} $\qU_n^{\imath}$ is the $\mathbb{A}$-subalgebra of
$\qU_{2n}$ generated by
\begin{align*}
e_i&=E_i+K_i^{-1}F_{-i},\qquad
f_i=E_{-i}+F_i K_{-i}^{-1},\quad (0<i<n),\\
t_0&=E_0+qF_0 K_0^{-1}+K_0^{-1},
\qquad d_j^{\pm 1}=(D_{j-\half} D_{-j+\half})^{\pm 1}, \quad(0<j \leq n).
\end{align*}
We have
\begin{align*}
\Delta(e_i)&=e_i \otimes K_i^{-1}+1 \otimes E_i+k_i^{-1} \otimes K_i^{-1}F_{-i},\\
\Delta(f_i)&=f_i \otimes K_{-i}^{-1}+1 \otimes E_{-i}+k_i \otimes F_i K_{-i}^{-1},\\
\Delta(d_j)&=d_j \otimes d_j,\\
\Delta(t_0)&=t_0 \otimes K_0^{-1}+1 \otimes (E_0+qF_0 K_0^{-1}).
\end{align*}
Therefore $\qU_n^\imath$ is a right coideal of $\qU_{2n}$.  Its specialization at $q\to1$ is $U(\mathfrak{gl}_{n}\oplus \mathfrak{gl}_n)$.

A Beilinson-Lusztig-MacPherson type realization of $\qU_n^\imath$ has been given in
\cite[\S Appendix]{BKLW18}. It induces a surjective $\mathbb{A}$-algebra homomorphism
$\kappa_{n,d}^\imath: \qU_n^\imath \twoheadrightarrow {}_\mathbb{A}\!\Sc_{n,d}^\imath$ satisfying
\begin{align}\label{kappai}
e_i &\mapsto \sum_{Z \in \Xi_{n,d-1}^{\diag,\imath}} [E_{i,i+1}^\theta+Z], \qquad f_i
\mapsto \sum_{Z \in \Xi_{n,d-1}^{\diag,\imath}} [E_{i+1,i}^\theta+Z], \quad (0<i<n),\\\nonumber
t_0 &\mapsto \sum_{Z \in \Xi_{n,d-1}^{\diag,\imath}} [E_{1,-1}^\theta+Z]+\sum_{Z \in
\Xi_{n,d}^{\diag,\imath}} q^{z_{11}}[Z],\quad
d_j \mapsto \sum_{Z \in \Xi_{n,d}^{\diag,\imath}} q^{z_{jj}}[Z], \quad (0<j \leq
n),
\end{align}
where
$\Xi_{n,d}^{\diag,\imath}=\{A \in \Xi_{n,d}^\imath ~|~ \mbox{$A$ is diagonal}\}$.

\subsection{Double centralizer property}
Now we can lift the general quantum Schur duality showed in Theorem~\ref{duality} from quantum Schur algebras to $\imath$quantum groups via the homomorphisms $\kappa_{n,d}^\jmath$ and $\kappa_{n,d}^\imath$.
\begin{thm}[geometric $\imath$Howe duality]
\label{doubleAIV}
	The actions
	$$\qU_m^\mathfrak{b} \quad \stackrel{\kappa_{m,d}^\mathfrak{b}}\twoheadrightarrow \quad
{}_\mathbb{A}\!\Sc_{m,d}^\mathfrak{b} \quad \stackrel{{}_\mathbb{A}\Phi}\curvearrowright \quad
{}_\mathbb{A}\!\T_{m|n,d}^{\mathfrak{bc}} \quad \stackrel{{}_\mathbb{A}\Psi}\curvearrowleft \quad
{}_\mathbb{A}\!\Sc_{n,d}^\mathfrak{c} \quad \stackrel{\kappa_{n,d}^\mathfrak{c}}\twoheadleftarrow
\quad \qU_n^\mathfrak{c},\quad(\mathfrak{b,c}\in\{\imath,\jmath\})$$
	satisfy
	$$ {}_\mathbb{A}\Phi\circ\kappa_{m,d}^\mathfrak{b}(\qU_m^\mathfrak{b})\cong
\End_{\qU_n^\mathfrak{c}}({}_\mathbb{A}\!\T_{m|n,d}^{\mathfrak{bc}}), \quad
\End_{\qU_m^\mathfrak{b}}({}_\mathbb{A}\!\T_{m|n,d}^{\mathfrak{bc}})\cong
{}_\mathbb{A}\Psi\circ\kappa_{n,d}^\mathfrak{c}(\qU_n^\mathfrak{c}).$$
\end{thm}

\subsection{Duality from geometry of type $\mC$}
Let $\mathrm{Sp}_{2d}$ be the symplectic group whose natural module $\F^{2d}$ is equipped with a non-degenerate skew-symmetric bilinear form $(\cdot,\cdot)$ satisying $(v_i,v_j)=\mathrm{sign}(i)\delta_{i,-j}$ for a given basis $\{v_{-d+\half},\ldots,v_{d-\half}\}$. The weight lattice $X$ for $\mathrm{Sp}_{2d}$
looks as the same as that for $\mathrm{SO}_{2d+1}$ in Subsection~\ref{weightB}. We write it again here:
$X=X^0 \sqcup X^{\half}$ where
$$X^0= \sum_{i=1}^d \mathbb{Z} \delta_i, \qquad X^{\half}= \sum_{i=1}^d(\half+\mathbb{Z})
\delta_i.$$
The Weyl group of type $C_d$, which is isomorphic to the one of type $B_d$, acts on $X$ by permutating
$\delta_i$ and changing the signs of coefficients of $\delta_i$, too.
Take
\begin{align*}
X_n^0=\{\sum_{i=1}^d a_i\delta_i ~|~ a_i \in \mathbb{Z},|a_i| \leq n, \forall i\},\quad
X_n^{\half}=\{\sum_{i=1}^d a_i\delta_i ~|~ a_i \in \half+\mathbb{Z},|a_i|< n, \forall
i\}
\end{align*}
which can be indexed by
\begin{align*}
\Lambda_{n,d}^{\mC,\jmath}=&\{\gamma=(\gamma_{-n},\ldots,\gamma_{-1},2\gamma_0,\gamma_1,\ldots,\gamma_n)
~|~  \sum_{i=0}^n \gamma_i=d,\gamma_i=\gamma_{-i}\},\\
\Lambda_{n,d}^{\mC,\imath}=&\{\gamma\in\Lambda_{n,d}^{\mC,\jmath}~|~\gamma_0=0\},
\end{align*}respectively.
Precisely,
an orbit $\gamma\in\Lambda_{n,d}^{\mC,\jmath}$ consists of all
the weights $\sum_{i=1}^d a_i\delta_i \in X_n^0$ such that
$$\gamma_k=\sharp \{i ~|~  |a_i|=k, i=1,\ldots,d\}, \quad (k=0,1,\ldots,n),$$
while an orbit $\gamma\in\Lambda_{n,d}^{\mC,\imath}$ consists of all
the weights $\sum_{i=1}^d a_i\delta_i \in X_n^{\frac{1}{2}}$ such that
$$\gamma_k=\sharp \{i ~|~  |a_i|=k-\frac{1}{2}, i=1,\ldots,d\}, \quad (k=1,2,\ldots,n).$$
Denote
\begin{align*}
\sF_{n,d}^{\mC,\jmath}=&\{\mathfrak{f}=(0=V_{-n-\half}\subset V_{-n+\half}\subset\cdots\subset
V_{n+\half}=\F^{2d}) \in \sF_{N,2d} ~|~ V_i=V_j^{\perp},\mbox{ if } i+j=0\},\\
\sF_{n,d}^{\mC,\imath}=&\{\mathfrak{f}\in\sF_{n,d}^{\mC,\jmath}~|~ V_{-\frac{1}{2}}=V_\frac{1}{2}\}\subset \sF_{n,d}^{\mC,\jmath}.
\end{align*}
Set $W_{i}=\langle v_{-d+\half},\ldots, v_{i-\half}\rangle$.
For $\mathfrak{b}\in\{\imath,\jmath\}$ and $\gamma \in\Lambda_{n,d}^{\mC,\mathfrak{b}}$,
the parabolic subgroup $P_\gamma\subset\mathrm{Sp}_{2d}$ consists of the elements which
stabilizes the flag
$$\mathfrak{F}_\gamma:=(0=W_{-d} \subset W_{-d+\gamma_{-d}} \subset \ldots \subset
W_{d-\gamma_d} \subset W_d=\F^{2d}).$$

\begin{lem}
	\label{variety isomorphism of type C}
	As varieties,
	$$\bigsqcup_{\gamma \in \Lambda_{n,d}^{\mC,\mathfrak{b}}} \mathrm{Sp}_{2d}/P_{\gamma} \simeq
\sF_{n,d}^{\mC,\mathfrak{b}},\ (\mathfrak{b}=\imath,\jmath):\quad [g] \in \mathrm{Sp}_{2d}/P_\gamma \mapsto g\mathfrak{f}_\gamma.$$
\end{lem}
\begin{proof}
  The argument is almost as the same as the proof of Lemma~\ref{variety isomorphism of type B}.
\end{proof}

Denote
\begin{align*}
\Xi_{n,d}^{\mC,\jmath}=&\{(a_{ij})_{-n\leq i,j\leq n} \in \mathrm{Mat}_N(\mathbb{N}) ~|~ a_{ij}=a_{-i,-j}, \sum_{i,j} a_{ij}=2d\},\\
\Xi_{m|n,d}^{\mC,\jmath}=&\{(a_{ij})_{-m\leq i\leq m;-n\leq j\leq n} \in \mathrm{Mat}_{M\times N}(\mathbb{N}) ~|~ a_{ij}=a_{-i,-j}, \sum_{i,j} a_{ij}=2d\},\\
\Xi_{n,d}^{\mC,\imath}=&\{(a_{ij}) \in \Xi_{n,d}^{\mC,\jmath} ~|~ a_{0i}=a_{i0}=0\},\quad
\Xi_{m|n,d}^{\mC,\imath}=\{(a_{ij}) \in \Xi_{m|n,d}^{\mC,\jmath}~|~ a_{i0}=a_{0i}=0\},\\
\Xi_{m|n,d}^{\mC,\jmath\imath}=&\{(a_{ij}) \in \Xi_{m|n,d}^{\mC,\jmath}~|~ a_{i0}=0\},\quad
\Xi_{m|n,d}^{\mC,\imath\jmath}=\{(a_{ij}) \in \Xi_{m|n,d}^{\mC,\jmath}~|~ a_{0i}=0\}.
\end{align*}
Let $\mathrm{Sp}_{2d}$ act diagonally on the products
$\sF_{m,d}^{\mC,\mathfrak{b}} \times \sF_{n,d}^{\mC,\mathfrak{c}}$, $(\mathfrak{b,c}\in\{\imath,\jmath\})$. Similar to the case of type B, we have the following bijection:
$$\mathrm{Sp}_{2d} \backslash \sF_{m,d}^{\mC,\mathfrak{b}} \times \sF_{n,d}^{\mC,\mathfrak{c}} \longleftrightarrow
\Xi_{m|n,d}^{\mC,\mathfrak{bc}}.$$

We set
$$\T_{m|n,d}^{\mC,\mathfrak{bc}}=\A_{\mathrm{Sp}_{2d}}(\sF_{m,d}^{\mC,\mathfrak{b}} \times \sF_{n,d}^{\mC,\mathfrak{c}})
\quad\mbox{and}\quad
\Sc_{n,d}^{\mC,\mathfrak{b}}=\T_{n|n,d}^{\mC,\mathfrak{bc}},\quad (\mathfrak{b,c}\in\{\imath,\jmath\}).$$
For $\mathfrak{b,c}\in\{\imath,\jmath\}$ and $A \in \Xi_{n,d}^{\mC,\mathfrak{b}}$ (resp. $\Xi_{m,d}^{\mC,\mathfrak{b}}$ and $\Xi_{m|n,d}^{\mC,\mathfrak{bc}}$), let $\chi_A
\in \Sc_{n,d}^{\mC,\mathfrak{b}}$ (resp. $\Sc_{m,d}^{\mC,\mathfrak{b}}$ and $\T_{m|n,d}^{\mC,\mathfrak{bc}}$)
be the characteristic function of the $\mathrm{Sp}_D$-orbit in $\sF_{n,d}^{\mC,\mathfrak{b}} \times
\sF_{n,d}^{\mC,\mathfrak{b}}$ (resp. $\sF_{m,d}^{\mC,\mathfrak{b}} \times \sF_{m,d}^{\mC,\mathfrak{b}}$ and $\sF_{m,d}^{\mC,\mathfrak{b}} \times
\sF_{n,d}^{\mC,\mathfrak{c}}$) associated with $A$. Denote
\begin{equation}\label{[A]c}
[A]:=q^{\half(\sum_{i \geq k,j<l}a_{ij}a_{kl}+\sum_{i \geq 0, j<0} a_{ij})} \chi_A.
\end{equation}
One can check that the $q$-power in \eqref{[A]c} matches the one in \eqref{[A]b} under the map $$\Xi_{n,d}^{\mC,\mathfrak{b}}\rightarrow\Xi_{n,d}^{\mathfrak{b}}: \quad A\mapsto A+E_{00}, \quad (\mathfrak{b}=\imath,\jmath).$$
Thanks to \cite[Proposition~6.7]{BKLW18}, we have the $\A$-algebra isomorphisms
$$\Sc_{n,d}^{\mC,\mathfrak{b}} \cong \Sc_{n,d}^\mathfrak{b}, \quad [A] \mapsto [A+E_{00}], \quad (\mathfrak{b}=\imath,\jmath).$$ Moreover, there are $\A$-module isomorphisms
$$\T_{m|n,d}^{\mC,\mathfrak{bc}} \cong \T_{m|n,d}^{\mathfrak{bc}}, \quad [A] \mapsto [A+E_{00}], \quad (\mathfrak{b,c}\in\{\imath,\jmath\}).$$ which are compatible with the actions of $\imath$Schur algebras:
$$\begin{matrix}
\label{iso of B,C}
{}_\mathbb{A}\!\Sc_{m,d}^{\mathfrak{b}} & \curvearrowright & {}_\mathbb{A}\!\T_{m|n,d}^{\mathfrak{bc}} &
\curvearrowleft & {}_\mathbb{A}\!\Sc_{n,d}^\mathfrak{c}\\
\updownarrow\cong & \quad & \updownarrow\cong & \quad & \updownarrow\cong\\
{}_\mathbb{A}\Sc_{m,d}^{\mC,\mathfrak{b}} & \curvearrowright &
{}_\mathbb{A}\T_{m|n,d}^{\mC,\mathfrak{bc}} & \curvearrowleft &
{}_\mathbb{A}\Sc_{n,d}^{\mC,\mathfrak{c}},
\end{matrix} \qquad (\mathfrak{b,c}\in\{\imath,\jmath\}).$$
It can be lift to $\imath$quantum groups, too. That is, the actions
$$\qU_m^\mathfrak{b} \quad \twoheadrightarrow \quad
{}_\mathbb{A}\!\Sc_{m,d}^{\mC,\mathfrak{b}} \quad \curvearrowright \quad
{}_\mathbb{A}\!\T_{m|n,d}^{\mC,\mathfrak{bc}} \quad \curvearrowleft \quad
{}_\mathbb{A}\!\Sc_{n,d}^{\mC,\mathfrak{c}} \quad \twoheadleftarrow
\quad \qU_n^\mathfrak{c},\quad(\mathfrak{b,c}\in\{\imath,\jmath\})$$
satisfy double centralizer property.

\section{Quantum coordinate coalgebras}\label{section5}

\subsection{Quantum coordinate coalgebra $\mathcal{T}_n^\jmath$}
Let $$\qU_n^{\jmath \circ}:=\{f \in \qU_n^{\jmath*}  ~|~  \mbox{$\Ker f$ contains a
cofinite ideal of $\qU_n^\jmath$}\}$$ be the cofinite dual of $\qU_n^\jmath$, which is
equipped with a coalgebra (the comultiplication is denoted by $\Delta^{\jmath\circ}$) and
right $\qU_N^\circ$-module structure, induced by the algebra and right coideal
structure of $\qU_n^\jmath$.
We remark that there is no multiplication on $\qU_n^{\jmath \circ}$ since there is no
comultiplication on $\qU_n^{\jmath}$.

Denote by $\qU_N^\circ$ the cofinite dual of $\qU_N$ as defined in Subsection~\ref{qca} but with a shift by $-n-1$ on the index of $t_{ij}$.
For $f,f'\in \qU_N^\circ$, denote $\widetilde{f}=f|_{\qU_n^\jmath}$ and $\widetilde{f\otimes f'}=f\otimes
f'|_{\qU_n^\jmath\otimes \qU_n^\jmath}$. It is clear that $\tilde{f} \in
\qU_n^{\jmath\circ}$.
\begin{lem}
	For any $x \in \qU_n^\jmath$, $f,f' \in \qU_N^\circ$ and $g \in \qU_n^{\jmath\circ}$,
we have
	\begin{align*}
	\tilde{f} f'&=\widetilde{ff'}, \quad
	\widetilde{\Delta^{\circ} (f)}=\Delta^{\jmath\circ}(\tilde{f}),\quad
	\Delta^{\jmath\circ} (gf)=\sum_{(g),(f)} (g_{(1)} f_{(1)}) \otimes (g_{(2)}
f_{(2)}),\\
	\langle gf,x \rangle&=\sum_{(x)} \langle g,x_{(1)} \rangle\langle f,x_{(2)} \rangle,
	\end{align*}
	where $\Delta^{\jmath\circ}(g)=\sum_{(g)} g_{(1)} \otimes g_{(2)}$ and
$\Delta^{\circ}(f)=\sum_{(f)} f_{(1)} \otimes f_{(2)}$.
	
	In particular, we have $\Delta^{\jmath\circ}(\tilde{t}_{ij})=\sum_k \tilde{t}_{ik}
\otimes \tilde{t}_{kj}$.
\end{lem}
\begin{proof}
	It just follows from basic properties of bialgebras.
\end{proof}

Recall the quantum coordinate algebra $\mathcal{T}_N$ of $\qU_N$ with unit element
$\ve$ (i.e. the counit of $\qU_N$). Notice that now the index set of $t_{ij}$ is $\{(i,j)~|~i,j=-n,-n+1,\ldots,n\}$. Let $\mathcal{T}_n^\jmath$ be the right cyclic
$\mathcal{T}_N$-module generated by $\tilde{\ve}$. It is easy to check that
$\mathcal{T}_n^\jmath$ has a coalgebra structure (but no algebra structure). We call
$\mathcal{T}_n^\jmath$ the \emph{quantum coordinate coalgebra} of $\qU_n^\jmath$.

\begin{lem}\label{Uj1}
	\begin{itemize}
		\item[(1)]
		The quantum coordinate coalgebra $\mathcal{T}_n^\jmath$ admits a
$\qU_n^\jmath$-bimodule structure via the following left and right actions:
		\begin{align*}
		x \cdot f:=\sum_{(f)} f_{(1)}\langle f_{(2)},x\rangle,\quad
		f \cdot x:=\sum_{(f)} \langle f_{(1)},x\rangle f_{(2)},
		\end{align*}
		where $x \in \qU_n^\jmath$, $f \in \mathcal{T}_n^\jmath$,
$\Delta^{\jmath\circ}(f)=\sum_{(f)} f_{(1)} \otimes f_{(2)}$.
		\item[(2)]
		The action of $\mathcal{T}_N$ on $\mathcal{T}_n^\jmath$ is a
$\qU_n^\jmath$-bimodule homomorphism from $\mathcal{T}_n^\jmath \otimes \mathcal{T}_N$
to $\mathcal{T}_n^\jmath$.
	\end{itemize}
\end{lem}
\begin{proof}
	The first statement is a basic property of coalgebras.

		For $x \in \qU_n^\jmath, f\in \mathcal{T}_N$ and $g \in \mathcal{T}_n^\jmath$, we have
	\begin{align*}
	x(gf)=\sum_{(g,f)}(g_{(1)} f_{(1)})\langle g_{(2)}f_{(2)},x \rangle=\sum_{(g,f,x)}g_{(1)} f_{(1)}\langle g_{(2)},x_{(1)} \rangle \langle
f_{(2)},x_{(2)} \rangle=\sum_{(x)} (x_{(1)} g)(x_{(2)} f),
	\end{align*}
	which verifies the second statement.
\end{proof}

\subsection{Basis theorem of $\mathcal{T}_n^\jmath$}
Thanks to the Schur duality between ${}_\mathbb{A}\!\HH(W_{B_d})$ and $\qU_n^\jmath$ on
$(\mathbb{A}^N)^{\otimes d}$ (cf. \cite{BW18}), we can obtain that
for $0<i,j \leq n$,
\begin{align}\label{rel:tij}
\tilde{t}_{ij}=\tilde{t}_{-i,-j}+(q-q^{-1})\tilde{t}_{i,-j},\quad
\tilde{t}_{i,-j}=\tilde{t}_{-i,j},\quad
\tilde{t}_{i0}=q\tilde{t}_{-i,0},\quad
\tilde{t}_{0j}=q\tilde{t}_{0,-j}.
\end{align}

Recall $\Xi_{n,d}^\jmath$ in \eqref{xind} and denote $$\Xi_n^{\jmath}=\bigsqcup_{d=0}^\infty\Xi_{n,d}^{\jmath}.$$
For $A\in\Xi_n^{\jmath}$, denote
$\tilde{t}^{(A)}=\tilde{\ve}\prod_{(i,j)
\geq (0,0)}^{<} t_{ij}^{a_{ij}^\sharp}$ where $a_{ij}^\sharp$ has been defined in \eqref{asharp}.
\begin{thm}
	\label{basis}
	The set $\{\tilde{t}^{(A)} ~|~ A \in \Xi_n^\jmath\}$ forms an $\mathbb{A}$-basis of
$\mathcal{T}_n^\jmath$.
\end{thm}
\begin{proof}
	Firstly we will show that the set spans $\mathcal{T}_n^\jmath$. For this purpose, we
just need to show that $\{\tilde{t}^{(A)} ~|~ A \in \Xi_{n,d}^\jmath \}$ spans
$\mathcal{T}_{n,d}^\jmath$ where $\mathcal{T}_{n,d}^\jmath=\tilde{\ve}\mathcal{T}_{N,d}$.
It holds for $d=1$ because of \eqref{rel:tij}. Assume the statement holds for $d-1$ and
we shall prove the case of $d$. Owing to $\mathcal{T}_{N,d}=\mathcal{T}_{N,d-1}
\mathcal{T}_{N,1}$, it suffices to show $\tilde{t}^{(B)} t_{ij}\in \langle\tilde{t}^{(A)}
~|~ A \in \Xi_{n,d}^\jmath \rangle$ for any $B\in\Xi_{n,d-1}^\jmath$, which is clear if $(i,j)=(n,n)$
since $\tilde{t}^{(B)} t_{nn}=\tilde{t}^{(B+E_{nn}^\theta)}$. We take $(k,l)$ to be the
maximal such that the $(k,l)$-th entry of $B$ is nonzero. Then either $\tilde{t}^{(B)}
t_{ij}=\tilde{t}^{(B+E_{ij}^\theta)}$ (if $(i,j)\geq(k,l)$) or $\tilde{t}^{(B)}
t_{ij}\in \mathcal{T}_{n,d-1}^\jmath t_{kl}+\mathcal{T}_{n,d-1}^\jmath t_{kj}$ (if
$(i,j)<(k,l)$) thanks to \eqref{tij}. Therefore
$\tilde{t}^{(B)} t_{ij} \in \langle\tilde{t}^{(A)} ~|~ A \in \Xi_{n,d}^{\jmath} \rangle$ via
recursion on $(k,l)$.
	
Next let us show that $\{\tilde{t}^{(A)} ~|~ A \in \Xi_n^\jmath \}$ are linearly independent.
	Denote $\mathcal{P}$ the coordinate algebra of $U(\mathfrak{gl}_{n+1} \oplus
\mathfrak{gl}_n)$ generated by the matrix elements $x_{ij}$ of the representation
$\K^N$, where $\K^N$ is the natural representation of $\mathfrak{gl}_N$. It
is not difficult to see
	$$\mathcal{P} \cong \K[x_{ij}|-n \leq i,j \leq n]/(x_{ij}-x_{-i,-j}).$$
	Write $x^{(A)}:=\prod_{(i,j) \geq (0,0)} x_{ij}^{a_{ij}^\sharp}$ ($A \in \Xi_n^\jmath$),
which form a basis of $\mathcal{P}$.
	
	Let $\sum_{A \in \Xi_n^\jmath} \kappa_A \tilde{t}^{(A)}=0$ be a finite sum. We may assume
that $\kappa_A \in \mathbb{K}[q]$ but not all $\kappa_A \in (q-1)\mathbb{K}[q]$
(multiplying a $q$-fraction if necessary). But $0=(\sum_{A \in \Xi_n^\jmath}
\kappa_A\tilde{t}^{(A)})|_{q=1}=\sum_{A \in \Xi_n^\jmath} \kappa_A|_{q=1} x^{(A)}$ implies
$\kappa_A|_{q=1}=0$ for all $\kappa_A$, a contradiction to our assumption. Therefore
$\{\tilde{t}^{(A)} ~|~ A \in \Xi_n^\jmath \}$ must be linearly independent. Hence it is a basis of $\mathcal{T}_n^\jmath$.
\end{proof}

\begin{cor}
  As coalgebras, $\mathcal{T}_n^\jmath\cong\mathcal{T}_N/\mathcal{I}$ where $\mathcal{I}$ is the right ideal of $\mathcal{T}_N$ generated by
  $$t_{ij}-t_{-i,-j}+(q^{-1}-q)t_{i,-j},\quad t_{i0}-qt_{-i,0},\quad t_{0,j}-qt_{0,-j},\quad t_{i,-j}-t_{-i,j}, \quad (0<i,j\leq n).$$
\end{cor}
The above corollary shows that our quantum coordinate coalgebra coincides with the one introduced by Lai-Nakano-Xiang (see \cite[Proposition~2.4.4]{LNX22}).

\subsection{Quantum coordinate coalgebra $\mathcal{T}_n^\imath$}
Denote by $\mathcal{T}_{2n}$ the quantum coordinate algebra of $\qU_{2n}$ as defined in Subsection~\ref{qca} but with a shift by $-n-\half$ on the index, e.g. $$t_{ij}, \quad (i,j=-n+\half,-n+\frac{3}{2},\ldots,n-\half).$$
Recall the quantum coordinate algebra $\mathcal{T}_{N}$ of $\qU_{N}$ and its elements
$$t_{ij}, \quad (i,j=-n,-n+1,\ldots,n).$$
There exists a bialgebra epimorphism $\varrho: \mathcal{T}_{N}\rightarrow\mathcal{T}_{2n}$ determined by
$$t_{ij}\mapsto t_{i-\frac{\mathrm{sign}(i)}{2},j-\frac{\mathrm{sign}(j)}{2}},\quad t_{i0}\mapsto 0,\quad t_{0j}\mapsto0,\quad t_{00}\mapsto\ve.$$ Thus $\mathcal{T}_{N}$ admits a $\qU_{2n}$-bimodule algebra structure by
$$x \cdot f=\sum_{(f)} f_{(1)}\langle x, \varrho(f_{(2)})\rangle, \qquad f \cdot x=\sum_{(f)} \langle x, \varrho(f_{(1)})\rangle f_{(2)}.$$
Therefore, we shall identify $\mathcal{T}_{2n}$ with the subalgebra of $\mathcal{T}_{N}$ generated by $\{t_{ij} ~|~ i,j \neq 0\}$ as a $\qU_{2n}$-bimodule algebra.

Recall $\Xi_{n,d}^\imath$ in \eqref{xind} and denote $$\Xi_n^{\imath}=\bigsqcup_{d=0}^\infty\Xi_{n,d}^{\imath}.$$
We denote by $\mathcal{T}_n^\imath$ the $\mathcal{T}_{2n}$-module with basis $\{\tilde{t}^{(A)} ~|~ A \in \Xi_n^\imath\}$. It is clear that $\mathcal{T}_n^\imath \cong \mathcal{T}_{2n}/(\mathcal{T}_{2n} \bigcap \mathcal{I})$. So $\mathcal{T}_n^\imath$ can be regarded as a subspace of $\mathcal{T}_n^\jmath$. Actually, $\mathcal{T}_n^\imath$ is a $\qU_n^\imath$-bimodule (it is a special case of Proposition~\ref{Uj2}).

\subsection{The $\mathbb{A}$-space $\mathcal{V}_{m|n}^{\mathfrak{bc}}$}
Let $s=\max\{m,n\}$. Recall $\Xi_{m|n,d}^{\mathfrak{bc}}$ in \eqref{xind} and denote $$\Xi_{m|n}^{\mathfrak{bc}}=\bigsqcup_{d=0}^\infty\Xi_{m|n,d}^{\mathfrak{bc}},\quad (\mathfrak{b,c}\in\{\imath,\jmath\}),$$ which can be regarded as subsets of $\Xi_s^\jmath$ by the natural way. Let $\mathcal{V}_{m|n}^\mathfrak{bc}$ be the subspace of $\mathcal{T}_s^\jmath$ with basis $\{\tilde{t}^{(A)} ~|~ A \in \Xi_{m|n}^\mathfrak{bc}\}$.
\begin{prop}\label{Uj2}
Let $A \in \Xi_{m|n}^\mathfrak{bc}$ $(\mathfrak{b,c}\in\{\imath,\jmath\})$. For $e_i,f_i
	\in \qU_n^\mathfrak{c}$, we have
	\begin{align*}
	e_i \cdot \tilde{t}^{(A)}=&\sum_{-n \leq j \leq n;a_{j,i+1}>0} q^{\sum_{k>j} (a_{k,i+1}-a_{ki})} [a_{j,i+1}]\tilde{t}^{(A+E_{ji}^\theta-E_{j,i+1}^\theta)},\\
	f_i \cdot \tilde{t}^{(A)}=&\sum_{-n \leq j \leq 0;a_{ji}^\sharp>0} q^{\sum_{k<j} (a_{ki}-a_{k,i+1})} [a_{ji}-\delta_{0i}\delta_{0j}]\tilde{t}^{(A+E_{j,i+1}^\theta-E_{ji}^\theta)}\\
	&+\sum_{0<j \leq n;a_{ji}>0} q^{\sum_{k<j} (a_{ki}-a_{k,i+1})-\delta_{0i}}[a_{ji}]\tilde{t}^{(A+E_{j,i+1}^\theta-E_{ji}^\theta)}.
\end{align*} For $e_i,f_i
	\in \qU_m^\mathfrak{b}$, we have
\begin{align*}
	\tilde{t}^{(A)} \cdot e_i=&\sum_{-n \leq j \leq 0;a_{ij}^\sharp>0}q^{\sum_{k \geq j} (a_{i+1,k}-a_{ik})+1+\delta_{0i}}[a_{ij}-\delta_{0i}\delta_{0j}]\tilde{t}^{(A+E_{i+1,j}^\theta-E_{ij}^\theta)}\\
	&+\sum_{0<j \leq n;a_{ij}>0}q^{\sum_{k \geq j} (a_{i+1,k}-a_{ik})+1}[a_{ij}]\tilde{t}^{(A+E_{i+1,j}^\theta-E_{ij}^\theta)},\\
	\tilde{t}^{(A)} \cdot f_i=&\sum_{-n \leq j \leq n;a_{i+1,j}>0} q^{\sum_{k \leq j} (a_{ik}-a_{i+1,k})+1}[a_{i+1,j}]\tilde{t}^{(A+E_{ij}^\theta-E_{i+1,j}^\theta)}.
\end{align*}
Moreover, if $\mathfrak{c}=\imath$, then for $t_0\in\qU_n^\imath$, we have
	\begin{align*}t_0 \cdot \tilde{t}^{(A)}=&q^{\sum_{j>0} (a_{j1}-a_{j,-1})+a_{01}}\tilde{t}^{(A)}\\
&+\sum_{-n \leq j \leq n;a_{j1}>0} q^{\sum_{k>j} (a_{j1}-a_{j,-1})-a_{j0}+\sum_{k<0} \delta_{jk}} [a_{j1}]\tilde{t}^{(A+E_{j,-1}^\theta)-E_{j1}^\theta};
\end{align*}
if $\mathfrak{b}=\imath$, then for $t_0\in\qU_m^\imath$, we have
\begin{align*}
\tilde{t}^{(A)} \cdot t_0=&q^{\sum_{j>0} (a_{1j}-a_{-1,j})+a_{10}}\tilde{t}^{(A)}\\&
+\sum_{-n \leq j \leq n;a_{1j}>0} q^{\sum_{k>j} (a_{1j}-a_{-1,j})-a_{0j}+\sum_{k<0} \delta_{jk}} [a_{1j}]\tilde{t}^{(A+E_{-1,j}^\theta-E_{1j}^\theta)}.
	\end{align*}
	Thus the $\mathbb{A}$-space $\mathcal{V}_{m|n}^\mathfrak{bc}$ forms a $(\qU_n^\mathfrak{c},\qU_m^\mathfrak{b})$-module.
\end{prop}
\begin{proof}
  Since the index sets are different between $\qU_N$ and $\qU_{2n}$, we shall deal with type $\jmath$ as a sample in the computation below, the computation for type $\imath$ is almost the same. We give a detail computation for $f_i \cdot \tilde{t}^{(A)}$ as follows:
	\begin{align*}
	f_i \cdot
	\tilde{t}^{(A)}=&\tilde{\ve}((E_{-(i+\half)}+F_{i+\half}K_{-(i+\half)}^{-1}) \cdot
	\prod_{(j,k) \geq (0,0)}^{<} t_{jk}^{a_{jk}^\sharp})
	\\
	=&\sum_{\stackrel{0<j \leq n;}{a_{j,-i}>0}} q^{\sum_{k>j} (a_{k,-i}-a_{k,-i-1})}
	[a_{j,-i}]\tilde{t}^{(A+E_{j,-i-1}^\theta-E_{j,-i}^\theta)}\\
	&+\delta_{0i}q^{\sum_{k>0} (a_{k0}-a_{k,-1})-a_{00}^\sharp}[a_{00}^\sharp]\tilde{t}^{(A+E_{0,-1}^\theta-E_{00}^\theta)}\\
	&+q^{\sum_{0<k\leq n}(a_{k,-i}-a_{k,-i-1})+\delta_{0i}a_{00}^\sharp}
	\sum_{\stackrel{0 \leq j \leq n;}{a_{ji}^\sharp>0}} q^{\sum_{0 \leq k<j} (a_{ki}^\sharp-a_{k,i+1})}[a_{ji}^\sharp]\tilde{t}^{(A+E_{j,i+1}^\theta-E_{ji}^\theta)}
	\\
	=&\sum_{\stackrel{-m \leq j<0;}{a_{ji}>0}} q^{\sum_{k<j}
		(a_{ki}-a_{k,i+1})}[a_{ji}]\tilde{t}^{(A+E_{j,i+1}^\theta-E_{ji}^\theta)}\\
	&+\delta_{0i}q^{\sum_{k<0} (a_{k0}-a_{k1})-a_{00}^\sharp}[a_{00}^\sharp]\tilde{t}^{(A+E_{01}^\theta-E_{00}^\theta)}\\
	&+q^{\sum_{-n \leq k<0}(a_{ki}-a_{k,i+1})+\delta_{0i}a_{00}^\sharp} \sum_{\stackrel{0
			\leq j \leq n;}{a_{ji}^\sharp>0}} q^{\sum_{0 \leq k<j} (a_{ki}^\sharp-a_{k,i+1})}[a_{ji}^\sharp]\tilde{t}^{(A+E_{j,i+1}^\theta-E_{ji}^\theta)}
	\\
	=&\sum_{\stackrel{-n \leq j \leq 0;}{a_{ji}^\sharp>0}} q^{\sum_{k<j}(a_{ki}-a_{k,i+1})}[a_{ji}-\delta_{0i}\delta_{0j}]\tilde{t}^{(A+E_{j,i+1}^\theta-E_{ji}^\theta)}\\
	&+\sum_{\stackrel{0<j \leq n;}{a_{ji}>0}} q^{\sum_{k<j}(a_{ki}-a_{k,i+1})-\delta_{0i}}[a_{ji}]\tilde{t}^{(A+E_{j,i+1}^\theta-E_{ji}^\theta)}.
	\end{align*}
	The computation for the other formulas is similar.
\end{proof}

For any $A=(a_{ij})\in \Xi_{m|n}^\mathfrak{bc}$, denote
\begin{equation*}
  \langle A\rangle=q^{\frac{(\ro_0(A)-1)(\ro_0(A)+1)}{4}+\sum_{1 \leq i \leq m}
\frac{\ro_i(A)(\ro_i(A)+1)}{2}}
\frac{\tilde{t}^{(A)}}{[a_{00}-1]!!\prod_{(i,j)>(0,0)}[a_{ij}]!} \in
\mathcal{V}_{m|n}^\mathfrak{bc},
\end{equation*}
where $\ro_i(A)=\sum_{-n \leq j \leq n} a_{ij}$. With this new notation, we can rewrite
the above proposition as follows.
\begin{cor} Let $A \in \Xi_{m|n}^\mathfrak{bc}$, $(\mathfrak{b,c}\in\{\imath,\jmath\})$. For $e_i,f_i\in \qU_n^\mathfrak{c}$, we have
	\begin{align*}
	e_i \cdot \langle A\rangle=&\sum_{-m \leq j \leq m;a_{j,i+1}>0} q^{\sum_{k>j}
(a_{k,i+1}-a_{ki})}[a_{ji}+1]\langle A+E_{ji}^\theta-E_{j,i+1}^\theta\rangle,
	\\
	f_i \cdot \langle A\rangle=&\sum_{-m\leq j \leq 0;a_{ji}^\sharp>0} q^{\sum_{k<j}
(a_{ki}-a_{k,i+1})}[a_{j,i+1}+1]\langle A+E_{j,i+1}^\theta-E_{ji}^\theta\rangle\\
	&+\sum_{0<j \leq m;a_{ji}>0} q^{\sum_{k<j}
(a_{ki}-a_{k,i+1})-\delta_{0i}}[a_{j,i+1}+1]\langle
A+E_{j,i+1}^\theta-E_{ji}^\theta\rangle;
\end{align*}
and for $e_i,f_i\in \qU_m^\mathfrak{b}$, we have
\begin{align*}
	\langle A\rangle \cdot e_i=&\sum_{-n \leq j \leq 0;a_{ij}^\sharp>0} q^{\sum_{k<j}
(a_{ik}-a_{i+1,k})}[a_{i+1,j}+1]\langle A+E_{i+1,j}^\theta-E_{ij}^\theta\rangle\\
	&+\sum_{0<j \leq n;a_{ij}>0} q^{\sum_{k<j}
(a_{ik}-a_{i+1,k})-\delta_{0i}}[a_{i+1,j}+1]\langle
A+E_{i+1,j}^\theta-E_{ij}^\theta\rangle,
	\\
	\langle A\rangle \cdot f_i=&\sum_{-n \leq j \leq n;a_{i+1,j}>0} q^{\sum_{k>j}
(a_{i+1,k}-a_{ik})}[a_{ij}+1]\langle A+E_{ij}^\theta-E_{i+1,j}^\theta\rangle.
	\end{align*}
Moreover, if $\mathfrak{c}=\imath$, then for $t_0\in\qU_n^\imath$, we have
	\begin{align*}t_0 \cdot \langle A\rangle=&q^{\sum_{j>0} (a_{j1}-a_{j,-1})+a_{01}}\langle A\rangle\\
&+\sum_{-n \leq j \leq n;a_{j1}>0} q^{\sum_{k>j} (a_{j1}-a_{j,-1})-a_{j0}+\sum_{k<0} \delta_{jk}} [a_{j,-1}+1-\delta_{j0}]\langle A+E_{j,-1}^\theta-E_{j1}^\theta\rangle;
\end{align*}
if $\mathfrak{b}=\imath$, then for $t_0\in\qU_m^\imath$, we have
\begin{align*}
\langle A\rangle \cdot t_0=&q^{\sum_{j>0} (a_{1j}-a_{-1,j})+a_{10}}\langle A\rangle\\&
+\sum_{-n \leq j \leq n;a_{1j}>0} q^{\sum_{k>j} (a_{1j}-a_{-1,j})-a_{0j}+\sum_{k<0} \delta_{jk}} [a_{-1,j}+1-\delta_{j0}]\langle A+E_{-1,j}^\theta-E_{1j}^\theta\rangle.
	\end{align*}
\end{cor}

For $\mathfrak{b,c}\in\{\imath,\jmath\}$ and $d\in\mathbb{N}$, let $\mathcal{V}_{m|n,d}^\mathfrak{bc}$ be the subspace of $\mathcal{V}_{m|n}^\mathfrak{bc}$ spanned by $\{\tilde{t}^{(A)} ~|~ A \in \Xi_{m|n,d}^\mathfrak{bc}\}$.
\begin{thm}\label{equivalentj}
	As $(\qU_n^\mathfrak{c},\qU_m^\mathfrak{b})$-modules,
	$$\mathcal{V}_{m|n}^\mathfrak{bc} \cong{}_\mathbb{A}\!\T_{n|m}^{\mathfrak{cb}}, \quad \mathcal{V}_{m|n,d}^\mathfrak{bc} \cong{}_\mathbb{A}\!\T_{n|m,d}^{\mathfrak{cb}}: \quad \langle
A\rangle \mapsto [A'].$$
\end{thm}
\begin{proof}
Comparing the above corollary with Proposition~\ref{prop:actionsj} together with the
homomorphisms $\kappa_{n,d}^\mathfrak{b}\ (\mathfrak{b}=\imath,\jmath)$ in \eqref{homo:kj} \& \eqref{kappai}, we obtain the desired isomorphisms.
\end{proof}

\section{Multiplicity-free decompositions of $\imath$Howe dualities} \label{section6}

\subsection{Classical highest weight module of $\qU_n^\jmath$}
By \cite{KP11}, there are automorphisms (a braid group action) $T_i^\jmath$, $(1\leq
i<n)$, on $\qU_n^\jmath$:
\begin{align*}
T_i^\jmath(e_j)&=
\begin{cases}
-f_i k_i, &\mbox{if} \quad j=i\\
[e_i,e_j]_{-1}, &\mbox{if} \quad |j-i|=1\\
e_j, &\mbox{otherwise}
\end{cases},
\quad
T_i^\jmath(f_j)=
\begin{cases}
-k_i^{-1} e_i, &\mbox{if} \quad j=i\\
[f_j,f_i]_1, &\mbox{if} \quad |j-i|=1\\
f_j, &\mbox{otherwise}
\end{cases},
\\
T_i^\jmath(d_j)&=d_{s_i(j)}.
\end{align*}
where $[x,y]_a=xy-q^ayx$.

Denote
$$t_0=[e_0,f_0]_1-\frac{k_0-k_0^{-1}}{q-q^{-1}} \quad \mbox{and}\quad  t_i=T_i^\jmath
\cdots T_1^\jmath(t_0), \quad(0<i<n).$$

Let $$\mathbb{A}_1=\left\{\frac{f(q)}{g(q)}~\middle|~f(q),g(q)\in\K[q],
g(1)\neq0\right\}$$ be the localization of $\K[q]$ at $(q-1)$. A left (resp. right)
$\qU_n^\jmath$-module $M$ is called a left (resp. right) highest weight module of highest
weight $(\mathbf{a,b})=(a_1,\ldots,a_{n+1},b_1,\ldots,b_n) \in \Z^{n+1} \times
\mathbb{A}_1^n$ if there exists $v \in M$ such that
\begin{align*}
M&=\qU_n^\jmath v, \quad d_i v=q^{a_{i+1}} v, \quad t_i v=b_{i+1} v, \quad e_i v=0\\
(\mbox{resp.}\quad
M&=v \qU_n^\jmath, \quad v d_i=q^{a_{i+1}} v, \quad v t_i=b_{i+1} v, \quad vf_i=0).
\end{align*}
Its specialization at $q\to1$ is a highest weight $U(\mathfrak{gl}_{n+1} \oplus
\mathfrak{gl}_n)$-module with highest weight
\begin{equation}\label{weightj}
  (a_1, \frac{a_2+b_1|_{q=1}}{2}, \ldots,
\frac{a_{n+1}+b_n|_{q=1}}{2},\frac{a_2-b_1|_{q=1}}{2}, \ldots,
\frac{a_{n+1}-b_n|_{q=1}}{2}).
\end{equation}
There exists a unique simple left (resp. right) highest weight module $L_\mathbf{a,b}^{[n],\jmath}$
(resp. $\widetilde{L}_\mathbf{a,b}^{[n],\jmath}$) of highest weight $(\mathbf{a,b})$ for any $(\mathbf{a,b}) \in \Z^{n+1} \times
\mathbb{A}_1^n$ (see \cite{Wa21}).

Let $$\Par_n^\jmath(d)=\bigsqcup_l \Par_{n+1}(d-l) \times \Par_n(l) \quad
\mbox{and} \quad \Par_n^\jmath=\bigsqcup_d \Par_n^\jmath(d).$$ For
$\lambda=(\lambda^+,\lambda^-) \in \Par_n^\jmath(d)$, denote $L_\lambda^{[n],\jmath}$ (resp.
$\widetilde{L}_\lambda^{[n],\jmath}$) the left (resp. right) irreducible highest weight
$\qU_n^\jmath$-module with highest weight
$$(q^{\lambda_1^+}, q^{\lambda_2^++\lambda_1^-}, \ldots, q^{\lambda_{n+1}^++\lambda_n^-},
[\lambda_2^+-\lambda_1^-], \ldots, [\lambda_{n+1}^+-\lambda_n^-]).$$
It is derived by \eqref{weightj} that at the specialization $q\to1$, $L_\lambda^{[n],\jmath}$ (resp.
$\widetilde{L}_\lambda^{[n],\jmath}$) specializes to the left (resp. right) irreducible $U(\mathfrak{gl}_{n+1} \oplus\mathfrak{gl}_n)$-module with highest weight $\lambda$, which we shall denote by $L_\lambda^{(n+1,n)}$ (resp. $\widetilde{L}_\lambda^{(n+1,n)}$).

\begin{lem}\cite[Theorem~4.3.7]{Wa20}\label{irrj}
If $L$ is a finite-dimensional irreducible $\qU_n^\jmath$-module on which $d_i$'s act
semisimply with eigenvalues in $\{q^a\}_{a \in \Z}$, then $L=L^{[n],\jmath}_\mathbf{a,b}$ for
some $(\mathbf{a,b}) \in \Z^{n+1} \times \mathbb{A}_1^n$ with $b_i=[k_i]$ for some $k_i
\in \Z$.
\end{lem}

\subsection{Classical highest weight module of $\qU_n^\imath$}
Now the automorphisms (braid group action) $T_i^\imath$, $(1 \leq i <n)$, on
$\qU_n^\imath$ are defined as follows:
\begin{align*}
T_i^\imath(e_j)&=
\begin{cases}
-f_i k_i, &\mbox{if} \quad j=i\\
[e_i,e_j]_{-1}, &\mbox{if} \qquad |j-i|=1\\
e_j,&\mbox{otherwise}
\end{cases},
\quad
T_i^\imath(f_j)=
\begin{cases}
-k_i^{-1} e_i, &\mbox{if} \quad j=i\\
[f_j,f_i]_{1}, &\mbox{if} \quad |j-i|=1\\
f_j, &\mbox{otherwise}
\end{cases},
\\
T_i^\imath(t_0)&=
\begin{cases}
[e_1,[t_0,f_1]_1]_{-1}+t_0 k_1, &\mbox{if} \quad i=1\\
t_0, &\mbox{otherwise}
\end{cases},
\qquad
T_i^\imath(d_j)=d_{s_i(j)},
\end{align*}

For $0<i<n$, denote
$$t_i=T_i^\imath \cdots T_1^\imath(t_0).$$
We remark that
\begin{equation}\label{dt=td}
t_id_j=d_jt_i,\quad(\forall 0\leq i\leq n, 1\leq j\leq n)
\end{equation}
because $t_0$ commutes with all $d_j (1\leq j\leq n)$ and $T_i^\imath$ $(1 \leq i <n)$ are automorphisms.

Similar to $\qU_n^\jmath$, we can define a left (resp. right) highest weight $\qU_n^\imath$-module of highest
weight $(\mathbf{a,b}) \in \Z^n \times \mathbb{A}_1^n$, whose specialization at $q\to1$ is a highest weight $U(\mathfrak{gl}_n \oplus \mathfrak{gl}_n)$-module with highest weight
\begin{equation}\label{weighti}
(\frac{a_1+b_1|_{q=1}-|\mathbf{a}|}{2}, \ldots, \frac{a_n+b_n|_{q=1}-|\mathbf{a}|}{2},
\frac{a_1-b_1|_{q=1}+|\mathbf{a}|}{2}, \ldots, \frac{a_n-b_n|_{q=1}+|\mathbf{a}|}{2}),
\end{equation} where $|\mathbf{a}|=\sum_{i=1}^na_i$.
Moreover, there also exists a unique left (resp. right) irreducible highest weight module $L_{\mathbf{a,b}}^{[n],\imath}$  (resp. $\widetilde{L}_\mathbf{a,b}^{[n],\imath}$) for $(\mathbf{a,b})$.

Let $$\Par_n^\imath(d)=\bigsqcup_l \Par_n(d-l) \times \Par_n(l) \quad
\mbox{and} \quad \Par_n^\imath=\bigsqcup_d \Par_n^\imath(d).$$
For $\lambda=(\lambda^+,\lambda^-) \in \Par_n^\imath(d)$, denote $L_\lambda^{[n],\imath}$ (resp.
$\widetilde{L}_\lambda^{[n],\imath}$) the left (resp. right) irreducible highest weight
$\qU_n^\imath$-module of the highest weight
$$(q^{\lambda_1^++\lambda_1^-}, \ldots, q^{\lambda_n^++\lambda_n^-},
[d+\lambda_1^+-\lambda_1^-], \ldots, [d+\lambda_n^+-\lambda_n^-]).$$
At the specialization $q\to1$, it follows from \eqref{weighti} that $L_\lambda^{[n],\imath}$ (resp. $\widetilde{L}_\lambda^{[n],\imath}$) specialize to the left (resp. right) irreducible $U(\mathfrak{gl}_n \oplus \mathfrak{gl}_n)$-module with highest weight $\lambda$, which we shall denote by $L_\lambda^{(n,n)}$ (resp. $\widetilde{L}_\lambda^{(n,n)}$).

Below is a $\qU^\imath$ counterpart of Lemma~\ref{irrj}.
\begin{lem}\label{irri}
If $L$ is a finite-dimensional irreducible $\qU_n^\imath$-module where $d_i$'s act
semisimply with eigenvalues in $\{q^a\}_{a \in \Z}$, and $t_0$ acts semisimply with
eigenvalues in $\{[k]~|~k \in \Z\}$, then $L=L^{[n],\imath}_\mathbf{a,b}$ for some
$(\mathbf{a,b}) \in \Z^n \times \mathbb{A}_1^n$ with $b_i=[k_i]$ for some $k_i \in \Z$.
\end{lem}
\begin{proof}
There must exists a common eigenvector subspace $M$ of $d_i, (1\leq i\leq n)$, which is killed by all $e_i, (1\leq i\leq n)$. This subspace $M$ is invariant under the action of $t_i$ by \eqref{dt=td}. It follows from \cite[Lemma~4.5.3]{Wa21} that all $t_i$-actions $(0\leq i\leq n)$ on $M$ commute. Thus we can find a singular vector in $M$, which is also a common eigenvector of $t_i, d_j, (0\leq i\leq n, 0<j\leq n)$. Here the existence of the highest weight vector $v$ depends on whether $t_i$ has eigenvalues in $\mathbb{A}_1$,  which is certified by \cite[Corollary~4.5.6]{Wa21}. Moreover, \cite[Corollary~4.5.6]{Wa21} implies that $L$ must be a irreducible highest weight $\qU_n^\imath$-module in the form described in the lemma.
\end{proof}

\subsection{Multiplicity-free decomposition} We shall give the multiplicity-free decomposition of the $(\qU_m^\mathfrak{b},\qU_n^\mathfrak{c})$-module ${}_\mathbb{A}\!\T_{m|n,d}^{\mathfrak{bc}}$ in this subsection.

\begin{lem}\label{eigen}
The left (resp. right) $t_0$-action on $\mathcal{V}_{m|n,d}^{\mathfrak{b}\imath}$ (resp. $\mathcal{V}_{m|n,d}^{\imath\mathfrak{b}}$), $(\mathfrak{b}=\imath,\jmath)$, is semisimple with eigenvalues in $\{[k+1] ~|~ -d \leq k \leq d\}$.
\end{lem}
\begin{proof}
	We just verify the case of left $t_0$-action on $\mathcal{V}_{m|n,d}^{\imath\imath}$ since the others are similar. Noting that $\mathcal{V}_{m|n,d}^{\imath\imath}$ is a quotient of $(\mathbb{A}^{2n} \otimes \mathbb{A}^{2m})^{\otimes d}$ as left $\qU_n^\imath$-modules, it is enough to calculate the eigenvalues of $t_0$ acting on $(\mathbb{A}^{2n})^{\otimes d}$.

When $d=1$, the eigenvalues lie in $\{[2],1,0\}$ via a straightforward computation by $t_0=E_0+qF_0 K_0^{-1}+K_0^{-1}\in\qU_{2n}$ and the natural $\qU_{2n}$-module structure explained in \eqref{naturalmoduleA} (notice that we take a shift on the index set when replace $n$ by $2n$).
Then using the comultiplication $\Delta(t_0)=t_0 \otimes K_0^{-1}+ 1 \otimes (E_0+qF_0 K_0^{-1})$, we can prove that the left $t_0$-action on $(\mathbb{A}^{2n})^{\otimes d}$ is semisimple with eigenvalues lying in $\{[k+1] ~|~ -d \leq k \leq d\}$ by induction on $d$.
\end{proof}

\begin{thm}\label{decomp}
	The $(\qU_m^\mathfrak{b},\qU_n^\mathfrak{c})$-module ${}_\mathbb{A}\!\T_{m|n,d}^{\mathfrak{bc}}\cong\mathcal{V}_{n|m,d}^\mathfrak{cb}$ $(\mathfrak{b,c}\in\{\imath,\jmath\})$ admits the following multiplicity-free decomposition:
	\begin{align*}
	{}_\mathbb{A}\!\T_{m|n,d}^{\mathfrak{bc}} &\cong\bigoplus_{\lambda \in \Par_m^\mathfrak{b}(d)\cap\Par_n^\mathfrak{c}(d)} L_{\lambda}^{[m],\mathfrak{b}} \otimes \widetilde{L}_\lambda^{[n],\mathfrak{c}}\\
&\cong
	\begin{cases}
	\bigoplus_{\lambda \in \Par_m^\mathfrak{b}(d)} L_{\lambda}^{[m],\mathfrak{b}} \otimes \widetilde{L}_\lambda^{[n],\mathfrak{c}}, &\mbox{if} \quad m < n,\\
	\bigoplus_{\lambda \in \Par_n^\mathfrak{c}(d)} L_{\lambda}^{[m],\mathfrak{b}} \otimes \widetilde{L}_\lambda^{[n],\mathfrak{c}}, &\mbox{if} \quad m> n,\\
	\bigoplus_{\lambda \in \Par_n^\mathfrak{\jmath}(d)} L_{\lambda}^{[m],\mathfrak{b}} \otimes \widetilde{L}_\lambda^{[n],\mathfrak{c}}, &\mbox{if} \quad m=n, \ \mathfrak{(b,c)}=(\jmath,\jmath),\\
\bigoplus_{\lambda \in \Par_n^\mathfrak{\imath}(d)} L_{\lambda}^{[m],\mathfrak{b}} \otimes \widetilde{L}_\lambda^{[n],\mathfrak{c}}, &\mbox{if} \quad m=n, \ \mathfrak{(b,c)}\neq(\jmath,\jmath).
	\end{cases}
	\end{align*}
\end{thm}
\begin{proof}
	The double centralizer property shown in Theorem~\ref{doubleAIV} implies that the $(\qU_m^\mathfrak{b},\qU_n^\mathfrak{c})$-module ${}_\mathbb{A}\!\T_{m|n,d}^{\mathfrak{bc}}\cong\mathcal{V}_{n|m,d}^\mathfrak{cb}$ has a multiplicity-free decomposition. By
Lemma~\ref{irrj} for $\jmath$ type or Lemmas~\ref{irri} \& \ref{eigen} for $\imath$ type, we know that each irreducible left $\qU_m^\mathfrak{b}$-module (resp. right $\qU_n^\mathfrak{c}$-module)
occurring in the decomposition must be in the form of $L_\lambda^{[m],\mathfrak{b}}$ (resp.
$\widetilde{L}_\lambda^{[n],\mathfrak{c}}$).
Thus the desired decomposition follows by the same multiplicity-free decomposition claim at the specialization $q\to1$, in which case $\qU_m^\jmath$ (resp. $\qU_n^\jmath$, $\qU_m^\imath$ and $\qU_n^\imath$) specializes to $U(\mathfrak{gl}_{m+1}\oplus \mathfrak{gl}_m)$ (resp. $U(\mathfrak{gl}_{n+1}\oplus \mathfrak{gl}_n)$, $U(\mathfrak{gl}_{m}\oplus \mathfrak{gl}_m)$ and $U(\mathfrak{gl}_{n}\oplus \mathfrak{gl}_n)$).

We take $(\mathfrak{b,c})=(\jmath,\imath)$ as a sample. At the specialization $q\to1$, the Fock space ${}_\mathbb{A}\!\T_{m|n,d}^{\jmath\imath}$ specializes to $((\mathbb{K}^M\otimes \mathbb{K}^{2n})^{\otimes d})^{W_{B_d}}$ as a $(\mathfrak{gl}_{m+1}\oplus \mathfrak{gl}_m, \mathfrak{gl}_{n}\oplus \mathfrak{gl}_n))$-module. It is known that as a $(\mathfrak{gl}_{m+1}\oplus \mathfrak{gl}_m, W_{B_d})$-module,
$$(\mathbb{K}^M)^d=\bigoplus_{\lambda\in\Par_m^\jmath(d)}L_\lambda^{(m+1,m)}\otimes\widetilde{ S}^\lambda,$$
and as a $(W_{B_d},\mathfrak{gl}_{n}\oplus \mathfrak{gl}_n)$-module,
$$(\mathbb{K}^{2n})^d=\bigoplus_{\lambda\in\Par_n^\imath(d)}S^\lambda \otimes \widetilde{L}_\lambda^{(n,n)},$$
where $S^\lambda$ and $\widetilde{S}^\lambda$ are the left and right irreducible $W_{B_d}$-modules corresponding to $\lambda$, respectively.
The above two formulas are very special examples of Schur dualities obtained in \cite{Hu01}.
Therefore, we have
\begin{align*}
  ((\mathbb{K}^M\otimes \mathbb{K}^{2n})^{\otimes d})^{W_{B_d}}&\cong((\mathbb{K}^M)^d\otimes (\mathbb{K}^{2n})^d)^{{W_{B_d}}} \qquad \mbox{(Here and below $W_{B_d}$ acts diagonally)}\\
  &\cong\bigoplus_{\stackrel{\lambda\in\Par_m^\jmath(d),}{\mu\in\Par_n^\imath(d)}}L_{\lambda}^{(m+1,m)}
  \otimes \widetilde{L}_{\mu}^{(n,n)}\otimes (\widetilde{S}^{\lambda}\otimes S^\mu)^{W_{B_d}}\\
  &\cong\bigoplus_{\lambda\in\Par_m^\jmath(d)\cap\Par_n^\imath(d)}L_{\lambda}^{(m+1,m)}
  \otimes \widetilde{L}_{\lambda}^{(n,n)}.
\end{align*} Other cases can be derived by the same argument. As we mentioned before, our desired multiplicity-free decompositions follow from these non-quantized ones.
\end{proof}


\end{document}